\documentclass[11pt]{article}
\usepackage{graphicx} 
\usepackage{geometry}
\usepackage{todonotes} 
\usepackage{graphicx} 
\usepackage{tikz}
\usetikzlibrary{patterns}
\usepackage{float} 
\usetikzlibrary{shapes}
\usetikzlibrary{decorations.pathreplacing}
\usetikzlibrary{decorations.pathmorphing}
\usetikzlibrary{positioning}
\usepackage{xcolor}
\usepackage{appendix}
\usepackage{enumitem}
\usepackage{natbib}
\usepackage{booktabs}
\usepackage{amsthm} 
\usepackage{amsmath,bm}
\usepackage{amssymb}
\usepackage{hyperref}
\usepackage{verbatim}
\usepackage[capitalise]{cleveref}
\newtheorem{theorem}{Theorem}[section]
\newtheorem{lemma}[theorem]{Lemma}
\newtheorem{observation}[theorem]{Observation}

\newtheorem{claim}[theorem]{Claim}
\newtheorem{proposition}[theorem]{Proposition}
\newtheorem{fact}[theorem]{Fact}

\theoremstyle{definition}
\newtheorem{remark}[theorem]{Remark}
\newtheorem{definition}[theorem]{Definition}
\newtheorem{construction}[theorem]{Construction}

\newenvironment{pf}[1][Proof]{%
    \begingroup
    \begin{proof}[#1]%
}{%
    \end{proof}
    \endgroup
}

\title{Stability with minuscule structure for chromatic thresholds}
\author{Jaehoon Kim\thanks{Department of Mathematical Sciences, KAIST, South Korea.        Email: \texttt{jaehoon.kim@kaist.ac.kr}. Supported by the National Research Foundation of Korea (NRF) grant funded by the Korea government(MSIT) No. RS-2023-00210430.}
\and
Hong Liu\thanks{Extremal Combinatorics and Probability Group (ECOPRO), Institute for Basic Science (IBS), Daejeon, South Korea. Email: \texttt{hongliu@ibs.re.kr}. Supported by IBS-R029-C4.}
\and
Chong Shangguan\thanks{Research Center for Mathematics and Interdisciplinary Sciences, Shandong University, and Frontiers Science Center for Nonlinear Expectations, Ministry of Education, Qingdao, China. Email: \texttt{theoreming@163.com}. Supported by National Natural Science Foundation of China under Grant No. 12231014.}
\and
Guanghui Wang\thanks{School of Mathematics, and  State Key Laboratory of Cryptography and Digital Economy Security, Shandong University, Jinan, China.  Email: \texttt{ghwang@sdu.edu.cn}. Supported by the Natural Science Foundation of China (No. 12231018).}
\and
Zhuo Wu\thanks{Departament de Matemàtiques, Universitat Politècnica de Catalunya (UPC),
Carrer de Pau Gargallo 14, 08028 Barcelona, Spain. Email: \texttt{zhuo.wu@upc.edu}. Supported by the bilateral AEI+DFG research project PCI2024-155080-2: SRC-ExCo – Structure, Randomness and Computational Methods in Extremal Combinatorics, and the PID2023-147202NB-I00 (COCOA: COntemporary COmbinatorics and its Applications), all funded by MICIU/AEI/10.13039/501100011033. }
\and
Yisai Xue\thanks{School of Mathematics and Statistics, Ningbo University, Ningbo, China and Extremal Combinatorics and Probability Group (ECOPRO), Institute for Basic Science (IBS), Daejeon, South Korea. Email: \texttt{xueyisai@nbu.edu.cn}. Supported by  National Natural Science Foundation of China (No. 12501486) and IBS-R029-C4.}}
\date{}

\begin{document}

\maketitle

\begin{abstract}
The chromatic threshold $\delta_\chi(H)$ of a graph $H$ is the infimum of $d>0$ such that the chromatic number of every $n$-vertex $H$-free graph with minimum degree at least $d n$ is bounded by a constant depending only on $H$ and $d$. Allen, B{\"o}ttcher, Griffiths, Kohayakawa, and Morris determined the chromatic threshold for every $H$; in particular, they showed that if $\chi(H)=r\ge 3$, then $\delta_\chi(H) \in\{\frac{r-3}{r-2},~\frac{2 r-5}{2 r-3},~\frac{r-2}{r-1}\}$. While the chromatic thresholds have been completely determined, rather surprisingly the structural behaviors of extremal graphs near the threshold remain unexplored. 

In this paper, we establish the stability theorems for chromatic threshold problems. We prove that every $n$-vertex $H$-free graph $G$ with $\delta(G)\ge (\delta_\chi(H)-o(1))n$ and $\chi(G)=\omega(1)$ must be structurally close to one of the extremal configurations.
Furthermore, we give a stronger stability result when $H$ is a clique, showing that $G$ admits a partition into independent sets and a small subgraph on sublinear number of vertices. 
We show that this small subgraph has fractional chromatic number $2+o(1)$ and is homomorphic to a Kneser graph defined by subsets of a logarithmic size set; both these two bounds are best possible. 
This is the first stability result that captures the lower-order structural features of extremal graphs.

We also study two variations of chromatic thresholds. Replacing chromatic number by its fractional counterpart, we determine the fractional chromatic thresholds for all graphs. Another variation is the bounded-VC chromatic thresholds, which was introduced by  Liu, Shangguan, Skokan, and Xu very recently. Extending work of {\L}uczak and Thomass{\'e} on the triangle case, we determine the bounded-VC chromatic thresholds for all cliques.
\end{abstract}

\section{Introduction} 
The \textit{chromatic number} of a graph $G$, denoted by $\chi(G)$, is the minimum number of colors needed to color its vertices so that adjacent vertices do not get the same color. The chromatic number is a central concept in graph theory. In the 1940s, Zykov \cite{zykov1949some} and Tutte \cite{tutte} first constructed triangle-free graphs with unbounded chromatic number. Later, Erd\H{o}s in 1959 \cite{erdos1959graph} extended this result, showing that for every $k,\ell\in\mathbb{N}$, there exist graphs with girth at least $\ell$ and chromatic number at least $k$ by a surprising probabilistic construction. 
However, all these constructions are sparse graphs with low minimum degrees relative to the number of vertices.

In 1973, Erd\H{o}s and Simonovits \cite{erdos1973valence} raised the problem of bounding the chromatic number of triangle-free graphs with large minimum degree. They introduced what is now known as the \emph{chromatic threshold} of a graph $H$:
\begin{align*}
\delta_\chi(H):=
& \inf \{d: \exists~C=C(H, d)\text{ such that if $G$ is a graph on $n$ vertices}, \\
& \text{with $\delta(G) \geq d n$ and $H \not \subseteq G$, then $\chi(G) \leq C$}\}.
\end{align*}
Based on a construction by Hajnal, Erd\H{o}s and Simonovits \cite{erdos1973valence} conjectured that $\delta_\chi(K_3)=1/3$. Over the past 50 years, the chromatic threshold problem has been one of the central topics in extremal graph theory. 
The conjecture of Erd\H{o}s and Simonovits was eventually confirmed in 2002 by Thomassen \cite{thomassen2002chromatic}. Later, Goddard and Lyle \cite{goddard2011dense} and Nikiforov \cite{nikiforov2010chromatic} independently extended the result of Thomassen by showing that $\delta_\chi(K_r)=\frac{2r-5}{2r-3}$ for every $r\ge 3$, and Thomassen \cite{thomassen2007chromatic} showed that the chromatic threshold of every odd cycle of length at least five is 0. 
 Further work on this topic can be found in \cite{brandt1999structure,chen1997triangle,haggkvist1982odd,jin1995triangle,luczak2010coloring,lyle2011chromatic}. 
Building upon ideas and tools developed by previous papers, in a remarkable work, Allen, B{\"o}ttcher, Griffiths, Kohayakawa, and Morris \cite{allen2013chromatic} completely determined the chromatic threshold for every $H$. In particular, they showed that for any graph $H$ with  $\chi(H)=r \geq 3$,  $\delta_\chi(H) \in\{\frac{r-3}{r-2},~\frac{2 r-5}{2 r-3},~\frac{r-2}{r-1}\}.$ We refer the interested readers to also~\cite{bourneuf2025denseneighborhoodlemmaapplications,huang2025interpolating,liu2024beyond} for more recent work related to this topic.

\subsection{Stability for chromatic thresholds}   
The main concern of this paper is to study the \textit{stability} of chromatic thresholds. Stability results demonstrate that graphs with nearly maximum size must be structurally close to extremal graphs. An archetype example of stability result is the one for Tur\'an problem. Let $T_{n,r-1}$ be the {\it Tur\'an graph}, which is the complete balanced $(r-1)$-partite graph on $n$ vertices. Tur\'an theorem \cite{Turan1941} shows that among all $n$-vertex $K_r$-free graphs, $T_{n,r-1}$ has the maximum number of edges, and it is the unique graph with this property.
In 1966, Erd\H{o}s and Simonovits~\cite{erdos1966limit,simonovitsmethod} famously proved that for any graph $H$ with $\chi(H)=r$, any $n$-vertex $H$-free graph $G$ with $e(G) \geq\big(\frac{r-2}{r-1}-o(1)\big)\binom{n}{2}$ can be obtained from $T_{n,r-1}$ by adding or deleting $o(n^2)$ edges. Originating from the work of Erd\H{o}s and Simonovits, stability has become a key topic and important tool in extremal graph theory and it has been extensively studied over the years; see \cite{Balogh_Clemen_Lavrov,furedi2015proof,illingworth2023minimum,kim2019asymptotic,liu2023stability,WangJian} for some recent examples. See also \cite{Tuan} for an example in additive combinatorics. 

In \cite{allen2013chromatic}, Allen et al. introduced three families of extremal graphs (see \cref{subsection2.1}), each corresponding to one value of $\delta_\chi(H)$ listed above (cf.~\cref{thm:morris}). Motivated by these highly structured constructions, we initiate a systematic study of the stability phenomena for chromatic thresholds. Although the chromatic threshold $\delta_\chi(H)$ has been completely determined, surprisingly little is known about the structure of extremal graphs near the threshold. Note that $\delta(G)\ge(\frac{r-2}{r-1}-o(1))n$ implies that $e(G) \geq\big(\frac{r-2}{r-1}-o(1)\big)\binom{n}{2}$, hence the stability theorem for the family of graphs $H$ with $\delta_\chi(H)=\frac{r-2}{r-1}$ follows from the Erd\H{o}s-Simonovits stability theorem. The main concern of this paper is to establish the stability theorems for graphs $H$ with $\delta_\chi(H)\in\big\{\frac{r-3}{r-2},~\frac{2 r-5}{2 r-3}\big\}$, as detailed below.

Let us begin with the most classical case of cliques. To discuss the extremal graphs for $\delta_{\chi}(K_r)$, we need to introduce the Kneser graphs. Let $n, m \in \mathbb{N}$. The \emph{Kneser graph} $\mathrm{KN}(n, m)$ is defined on the vertex set $\binom{[n]}{m}$, where two vertices $A,B\in\binom{[n]}{m}$ are adjacent if and only if $A\cap B=\varnothing$. Roughly speaking, the extremal graphs for $\delta_{\chi}(K_r)$ consists of two parts:
\begin{itemize}
    \item[(\emph{Micro})] one part is a Kneser graph $\mathrm{KN}(m,(\frac{1}{2}-o(1))m)$ on $o(n)$ vertices that contributes to the high chromatic numbner of $G$;
    \item[(\emph{Macro})] the other is a complete $(r-1)$-partite graph with parts ratio $(\frac{2}{2r-3},\ldots,\frac{2}{2r-3},\frac{1}{2r-3})$ that establishes the required minimum degree of $G$. 
\end{itemize} 
Notice that although Kneser graphs play a decisive role in establishing the lower bound of $\delta_{\chi}(K_r)$, they only account for a negligible $o(n)$ vertices in the construction (see~\cref{construction:K_r} for more details). 

Since the construction of Hajnal in the 60s, no other tight constructions are known. One can replace the use of Kneser graph in Hajnal's construction with an appropriate Borsuk graph. However, this is essentially the same construction as such Borsuk graph has independence number almost half of its order and admits a homomorphism to a Kneser graph $\mathrm{KN}(m,(\frac{1}{2}-o(1))m)$ for some $m$. It is not inconceivable that there could be other types of extremal graphs for cliques. 

Our main result reads as follows, showing that all \emph{almost} extremal graphs must essentially look like the one in Hajnal's construction, thereby establishing the stability for $\delta_{\chi}(K_r)$. The fractional chromatic number $\chi_f$ is defined in \cref{sec:preliminarlies}.
Let $K_{t}(n_1,\ldots,n_t)$ denote the complete $t$-partite graph with parts of sizes $n_1, \ldots, n_t$.

\begin{theorem}[Stability for $\delta_{\chi}(K_r)$]\label{thm:Kr}
Let $r\ge 3$ and $G$  be a $K_r$-free graph on $n$ vertices with $\delta(G)\ge\big(\frac{2r-5}{2r-3}-o(1)\big)n$ and $\chi(G)=\omega(1)$. Then,
\begin{equation}\label{eq:macro}
 \Big|E(G)\triangle E\Big(K_{r-1}\Big(\frac{2n}{2r-3},\ldots,\frac{2n}{2r-3},\frac{n}{2r-3}\Big)\Big)\Big|=o(n^2).    
\end{equation}
Furthermore, $G$ admits a vertex partition $V(G)=A^*\cup B^*_1\cup\cdots\cup B^*_{r-1}$ with the following properties:
\begin{itemize}
\item [{\rm (i)}] $|A^*|=o(n)$, $|B^*_i|=\big(\frac{2}{2r-3}\pm o(1)\big)n$ for each $i\in[r-2]$, and $|B^*_{r-1}|=\big(\frac{1}{2r-3}\pm o(1)\big)n$;
\item [{\rm (ii)}] for each $i\in[r-1]$, $B^*_i$ is an independent set;
\item [{\rm (iii)}] $\chi(G[A^*])=\omega(1)$ and $\chi_f(G[A^*])\le 2+o(1)$; moreover, $G[A^*]$ admits a homomorphism to $\mathrm{KN}(m,(\frac{1}{2}-o(1))m)$, where $m = O(\log n)$. 
\end{itemize}
\end{theorem}

Several remarks are in order. First of all,~\eqref{eq:macro} and parts~(i) and (ii) determine the macro structure of almost extremal graphs similar to other stability results. The interesting part of~\cref{thm:Kr} is that it also captures their micro structure as shown by part~(iii). 
This is the \emph{first} stability result that witnesses such lower scale structure. Moreover, in part~(iii), the bound on the fractional chromatic number on part $A^*$ is best possible. More importantly, the bound on $m$ is optimal. 
Indeed, by taking $m\geq \tfrac{1}{2}\log_2 n$ with $\binom{m}{ (\tfrac{1}{2}-o(1))m } \leq \sqrt{n}$, $G[A^*]$ could be a Kneser graph $\mathrm{KN}(m,(\frac{1}{2}-o(1))m)$. 
To achieve this optimal bound on $m$, we utilize random projections akin to the Johnson--Lindenstrauss lemma~\cite{johnson1984extensions} for low-distortion embedding of high dimensional point sets.

We also establish the stability result for general graphs $H$ with $\chi(H)=r\ge 3$ and $\delta_\chi(H)=\frac{2r-5}{2r-3}$. For a graph $H$ with $\chi(H) = r \geq 3$, the \emph{decomposition family} $\mathcal{M}(H)$ consists of all bipartite graphs that can be obtained from $H$ by deleting $r-2$ color classes in an $r$-coloring of $H$.

\begin{theorem}[Stability for $\delta_\chi(H)=\frac{2r-5}{2r-3}$]\label{thm:lambda-stability} 
Let $H$ be a graph with $\chi(H)=r\ge 3$ and $\delta_\chi(H)=\frac{2r-5}{2r-3}$.
Every $H$-free graph $G$ on $n$ vertices with $\delta(G)\ge\big(\frac{2r-5}{2r-3}-o(1)\big)n$ and $\chi(G)=\omega(1)$ satisfies
$$\Big|E(G)\triangle E\Big(K_{r-1}\Big(\frac{2n}{2r-3},\ldots,\frac{2n}{2r-3},\frac{n}{2r-3}\Big)\Big)\Big|=o(n^2). $$
Furthermore, $G$ admits a vertex partition $V(G)=A\cup B_1\cup\cdots\cup B_{r-1}$ with the following properties:
\begin{itemize}
\item [{\rm (i)}] $|A|=o(n)$, $|B_i|=\big(\frac{2}{2r-3}\pm o(1)\big)n$ for each $i\in[r-2]$, and $|B_{r-1}|=\big(\frac{1}{2r-3}\pm o(1)\big)n$;
\item [{\rm (ii)}] for every forest $F\in\mathcal{M}(H)$ and $i\in[r-1]$, $G[B_i]$ is $F$-free;
\item [{\rm (iii)}]  $\chi(G[A])=\omega(1)$ but $\chi_f(G[A])= O(1)$.
\end{itemize}
\end{theorem}

It would be interesting to see if the fractional chromatic number in part~(iii) can be improved to $2+o(1)$ as in the clique case.

Finally, we present the stability result for the graphs $H$ with $\delta_\chi(H)=\frac{r-3}{r-2}$.

\begin{theorem}[Stability for $\delta_\chi(H)=\frac{r-3}{r-2}$]\label{thm:theta-stability}
Let $H$ be a graph with $\chi(H)=r\ge 4$ and $\delta_\chi(H)=\frac{r-3}{r-2}$.
Every $H$-free graph $G$ on $n$ vertices with $\delta(G)\ge\big(\frac{r-3}{r-2}-o(1)\big)n$ and $\chi(G)=\omega(1)$ contains a vertex subset $S\subseteq V(G)$ of size $\frac{n}{r-2}$ such that the following holds. Let $\hat{G}$ be the spanning subgraph of $G$ obtained by deleting all edges induced by $S$, then
\begin{align*}
|E(\hat{G})\triangle E(T_{n,r-2})|= o(n^2).
\end{align*}
\end{theorem}

In comparison with \cref{thm:Kr} and \cref{thm:lambda-stability} where we nail down the location of $(1-o(1))$-fraction of edges, here in \cref{thm:theta-stability} we have no control on the edges induced by the special part $S$. Indeed, in the extremal graphs, the induced subgraph on part $S$ could be any `locally bipartite' graph with edge density between 0 and 1/2; we refer the readers to  \cref{remark} for more details.

\subsection{Other extensions of chromatic thresholds}
The chromatic threshold problem has been extended to what is now known as the \textit{homomorphism threshold}, as initially discussed in \cite{thomassen2007chromatic}. Further extensions have been discussed in \cite{ebsen2020homomorphism,goddard2011dense,huang2025interpolating}.
This section delves into two additional extensions of chromatic thresholds.

In Theorems \ref{thm:Kr} and \ref{thm:lambda-stability}, while the chromatic number of graph $G$ can be arbitrarily large, the fractional chromatic number remains bounded. This leads to the following natural question: can we reduce the minimum degree requirement of an $H$-free graph $G$ if only asking for bounded fractional chromatic number? More precisely, we define the \emph{fractional chromatic threshold} for a graph $H$ as follows:
\begin{align*}
\delta_{\chi_{f}}(H):=& \inf \{d: \exists~C=C(H, d)\text{ such that if $G$ is a graph on $n$ vertices}, \\
& \text{with $\delta(G) \geq d n$ and $H \not \subseteq G$, then $\chi_{f}(G) \leq C$}\}.
\end{align*}

In contrast to the usual chromatic thresholds (cf.~\cref{thm:morris}), Kneser graph is not an obstruction any more for cliques in the fractional counterpart. One shall expect a lower threshold in the fractional setting. Our next result confirms this intuition, determining the fractional chromatic threshold for all graphs.

\begin{theorem}\label{thm:fractional-chromatic-threshold-2}
Let $H$ be a graph with $\chi(H)=r\ge 3$. Then
\begin{align*}
\delta_{\chi_{f}}(H)=
  \begin{cases} 
  \frac{r-3}{r-2} & \text{if $\mathcal{M}(H)$ contains a forest},
  \\   \frac{r-2}{r-1} & \text{otherwise}.
\end{cases}
\end{align*}
\end{theorem}

While being interesting on its own,~\cref{thm:fractional-chromatic-threshold-2} is also a key ingredient in proving~\cref{thm:lambda-stability}(iii) when we need to bound the fractional chromatic number of the high chromatic part $A$.

Another extension is to combine with the notion of~\emph{VC-dimension} (see \cref{sec:bounded-VC} for definition).  The study of the use of VC-dimension in chromatic thresholds began in 2010, when {\L}uczak and Thomass{\'e} \cite{luczak2010coloring} provided an intriguing new proof of $\delta_\chi(K_3)=1/3$ using the theory of VC-dimension. Moreover, they showed that for graphs with bounded VC-dimension, the minimum degree condition can be relaxed to $\delta(G)\ge cn$ for any constant $c>0$. Recently, Liu, Shangguan, Skokan, and Xu \cite{liu2024beyond} gave a new short proof of {\L}uczak and Thomass{\'e}'s result and introduced the concept of \textit{bounded-VC chromatic thresholds}, as described below:
\begin{align*}
\delta_\chi^{\text{VC}}(H): = 
& \inf \{c\ge 0: \forall d\in \mathbb{N},~\exists~C=C(c,d,H) \text{ such that if $G$ is a graph on $n$ vertices},\\ 
& \text{with $\mathrm{VC}(G)\le d$, $\delta(G) \geq c n$ and $H \not \subseteq G$, then $\chi(G) \leq C$}. \}
\end{align*}

In this formulation, {\L}uczak and Thomass{\'e}'s result implies $\delta_\chi^{\text{VC}}(K_3)=0$. We extend their result by determining the bounded-VC chromatic thresholds for all complete graphs. 
\begin{theorem}\label{thm:bounded-VC}
For every $r\ge 3$,  $\delta_\chi^{\mathrm{VC}}(K_r)=\frac{r-3}{r-2}.$
\end{theorem}

We note that~\cref{thm:bounded-VC} has already found applications; it is used in very recent work of Huang, Liu, Rong and Xu~\cite{huang2025interpolating} where they introduced an asymmetric version of homomorphism thresholds that interpolates the bounded-VC chromatic and homomorphism threshold problems.


\paragraph{Structure of this paper.} In \cref{sec:preliminarlies}, we present some preliminaries including the construction of the extremal graphs for chromatic thresholds. We first give the proof of~\cref{thm:lambda-stability} in~\cref{sec:lambda-stability} and then build on it to prove~\cref{thm:Kr} in~\cref{sec:Kr-stability}. The proof of~\cref{thm:theta-stability} is given in~\cref{sesc:theta-stability}. Finally, we will determine the fractional chromatic thresholds in \cref{fractional} and the bounded-VC chromatic thresholds of cliques in \cref{sec:bounded-VC}.

\section{Preliminaries}\label{sec:preliminarlies}
\paragraph{Notation.}\label{subsec:defs}
For positive integers $k\le n$, let $[n]=\{1,\ldots,n\}$, and let $\binom{[n]}{k}$ denote the family of all $k$-element subsets of $[n]$. Given two finite sets $X,Y$, their \emph{symmetric difference} is denoted as $X\triangle Y=(X\setminus Y)\cup (Y\setminus X)$. For reals $a,b,c$, we write $a=(b\pm c)$ if $a\in[b-c,b+c]$. For the sake of presentation, we will omit floors and ceilings whenever they are not important.

Let $G = (V, E)$ be a graph, where $V$ denotes the set of vertices and $E$ denotes the set of edges. The number of vertices of $G$ is given by $|G|:=|V(G)|$, and the number of edges by $e(G) := |E(G)|$. 
For a vertex $v \in V(G)$, let $N(v)$ be the set of neighbors of $v$, and let $d(v)=|N(v)|$ be the \emph{degree} of $v$. 
The \emph{minimum degree} of $G$ is $\delta(G) = \min_{v \in V(G)} d(v)$. 
Given a subset $U \subseteq V(G)$, let $N(U)=\bigcap_{u\in U}N(u)$ denote the \emph{common neighborhood} of $U$.
The \emph{induced subgraph} $G[U]$ consists of the vertices in $U$ and all edges in $G$ that have both endpoints in $U$. For disjoint vertex sets $U_1,\ldots, U_s$, let $G[U_1,\ldots, U_s]$ denote the subgraph induced by  the edges with endpoints in distinct sets $U_i$ and $U_j$ for $1 \leq i < j \leq s$.

For two graphs $G$ and $H$, the \emph{join} $G \vee H$ is formed by taking disjoint copies of $G$ and $H$, and adding all edges between $V(G)$ and $V(H)$.  The $s$-blowup of a graph $G$, denoted by $G[s]$, is obtained by replacing each vertex $v \in V(G)$ with an independent set $I_v$ of size $s$, and connecting two sets $I_u$ and $I_v$ with a complete bipartite graph whenever $uv \in E(G)$.

Given two graphs $G$ and $F$, we say that $G$ is \emph{homomorphic} to $F$, denoted by $G\xrightarrow{\textup{hom}} F$, if there exists a homomorphism $\phi:V(G)\rightarrow V(F)$ that preserves adjacencies, that is, for every $uv\in E(G)$, we have $\phi(u)\phi(v)\in E(F)$.

For positive integer $a,b$, a \emph{$b$-fold coloring} of a graph $G$ is an assignment of sets of size $b$ to its vertices such that adjacent vertices receive disjoint sets; an \emph{$a:b$-coloring} is a $b$-fold coloring where each assigned set is chosen from $\binom{[a]}{b}$. Note that $G$ has an $a:b$-coloring if and only if $G\xrightarrow{\textup{hom}} \mathrm{KN}(a,b)$. The \emph{$b$-fold chromatic number} $\chi_b(G)$ of $G$ is the smallest $a$ such that there exists an $a:b$-coloring of $G$. The \emph{fractional chromatic number} $\chi_f(G)$ of $G$ is defined to be $ \chi_f(G)=\lim\limits_{b\rightarrow\infty}\frac{\chi_b(G)}{b}=\inf\limits_{b}\frac{\chi_b(G)}{b}.$  It is clear that we always have $\chi_f(G)\le \chi(G)$, and if $G\xrightarrow{\textup{hom}} \mathrm{KN}(a,b)$, then $\chi_f(G)\le a/b$.

\subsection{Extremal graphs for chromatic thresholds}\label{subsection2.1}

In this subsection, we review the three families of graphs constructed in \cite{allen2013chromatic}, which establish the lower bounds for chromatic thresholds.
Readers familiar with these constructions can skip this subsection.

In \cite{luczak2010coloring}, {\L}uczak and Thomass{\'e} defined a graph $H$ to be \emph{near-acyclic} if $\chi(H)=3$ and $H$ admits a partition into a forest $F$ and an independent set $S$ such that every odd cycle of $H$ intersects $S$ in at least two vertices. Equivalently, for each tree $T$ in $F$ with color classes $V_1(T)$ and $V_2(T)$, no vertex in $S$ has neighbors in both $V_1(T)$ and $V_2(T)$\footnote{For example, $C_5$ is near-acyclic.}. 
 We say that $H$ is \emph{$r$-near-acyclic} if $\chi(H) = r \geq 3$ and there exist $r-3$ independent sets in $H$ whose removal results in a near-acyclic graph.

\begin{theorem}[\cite{allen2013chromatic}]\label{thm:morris}
  Let $H$ be a graph with $\chi(H)=r \geq 3$. Then
\begin{align*}
\delta_\chi(H) \in\left\{\frac{r-3}{r-2},~\frac{2 r-5}{2 r-3},~\frac{r-2}{r-1}\right\}.
\end{align*}
 Moreover, $\delta_\chi(H) \neq \frac{r-2}{r-1}$ if and only if $H$ has a forest in its decomposition family, and $\delta_\chi(H)=$ $\frac{r-3}{r-2}$ if and only if $H$ is $r$-near-acyclic.
\end{theorem}

In \cite{allen2013chromatic}, the authors constructed three families of graphs, each corresponding to one of the values
$\delta_\chi(H)\in\{\frac{r-3}{r-2},\frac{2 r-5}{2 r-3},\frac{r-2}{r-1}\}$,
such that for every $C\in\mathbb{N}$ there exists an infinite sequence of $n$-vertex $H$-free graphs $G_n$ with minimum degree $\delta(G_n)\ge(\delta_\chi(H)-o(1))\cdot n$ and chromatic number $\chi(G_n)\ge C$, where $o(1)\rightarrow 0$ as $n\rightarrow \infty$. We refer to these graphs as \emph{extremal graphs} for chromatic thresholds.

\paragraph{Extremal graphs for $\delta_\chi(H)=\frac{r-2}{r-1}$.} 

The \emph{girth} of a graph $G$ is defined as the length of its shortest cycle. 
The original proof by Erd\H{o}s \cite{erdos1959graph} showed the existence of graphs $G$ in which both the girth and the ratio $|G| / \alpha(G)$ are arbitrarily large.
Consequently, it follows from $\chi(G)\ge \chi_f(G)\ge |G|/\alpha(G)$ \cite{scheinerman2013fractional} that the girth and the fractional chromatic number can also be made arbitrarily large simultaneously.


For each $k, \ell \in \mathbb{N}$, we shall call a graph $G$ a \emph{$(k, \ell)$-Erd\H{o}s graph} if it has fractional chromatic number at least $k$, and girth at least $\ell$.

\begin{construction}[Extremal graphs for $\delta_\chi(H)=\frac{r-2}{r-1}$]\label{construction:extremal-graph-pi}
  For every such $H$ and every $C \in \mathbb{N}$, let $G'$ be a $(C,|H|+1)$-Erd\H{o}s graph. Let $G=G'\vee K_{r-2}[|G'|]$. In other words, $G$ is obtained from placing a copy of $G'$ into one part of a balanced complete $(r-1)$-partite graph.  
\end{construction}

Note that $G$ is $H$-free as there is no forest in the decomposition family of $H$ and the girth of $G'$ is larger than the order of $H$, and $G$ satisfies $\delta(G)=\frac{r-2}{r-1}\cdot |G|$ and $\chi(G) \geq C$, thus confirming the lower bound $\delta_\chi(H)\ge\frac{r-2}{r-1}$.

\paragraph{Extremal graphs for $\delta_\chi(H)=\frac{2r-5}{2r-3}$.} In order to establish a lower bound for the family of graphs $H$ with $\chi(H)=r\ge3$ and $\delta_\chi(H)=\frac{2r-5}{2r-3}$, the authors of \cite{allen2013chromatic} introduced the \emph{$r$-Borsuk-Hajnal graph}, which is an extension of the \emph{$r$-Hajnal graph} (see \cref{construction:K_r}). We  utilize only the $r$-Hajnal graph. We recommend that readers refer to \cite{allen2013chromatic} for more details.

\begin{definition}[Hajnal graphs]\label{def:Hajnal-graph}
Let $k, \ell, m \in \mathbb{N}$ and $2 m+k\mid\ell$. The \emph{Hajnal graph} $H(k, \ell, m)$ is a graph on $\binom{2m+k}{m}+3 \ell$ vertices obtained as follows (see \cref{1.2} below for an illustration).
\begin{itemize}
\item Take vertex disjoint copies of a Kneser graph $\mathrm{KN}(2 m+k, m)$ and a complete bipartite graph $K_{2 \ell, \ell}$ with vertex set $A \cup B$, where $|A|=2 \ell$, and $|B|=\ell$;
\item Partition $A$ into $2 m+k$ pieces $A_1, \ldots, A_{2 m+k}$ of equal size;
\item For $S \in \binom{[2m+k]}{m}=V(\mathrm{KN}(2 m+k, m))$, $j\in [2m+k]$ and $y \in A_j$, add an edge between $S$ and $y$ whenever $j \in S$ (so every $S\in \binom{[2m+k]}{m}$ is connected to precisely $\frac{m|A|}{2m+k}$ vertices of $A$).
\end{itemize}  
\end{definition}

\begin{figure}[!ht]
    \centering
\begin{tikzpicture}[scale=0.8]
  \draw[line width=0.901pt] (0, 3) circle (1cm);
  \node[below] at (0,2) {\small $\mathrm{KN}(2m+k,m)$};
  \draw [red,line width=0.901pt](0.4, 3.3) -- (2.4, 5.2);
  \draw [red,line width=0.901pt](0.4, 3.3) -- (2.3, 4.6);
  \draw [red,line width=0.901pt](0.4, 3.3) -- (2.3, 3.7);
  \draw [blue,line width=0.901pt](0.4, 2.7) -- (2.3, 2.3);
  \draw [blue,line width=0.901pt](0.4, 2.7) -- (2.3, 1.4);
  \draw [blue,line width=0.901pt](0.4, 2.7) -- (2.4, 0.7);
  \filldraw (0.4, 3.3) circle (0.07cm);
  \filldraw (0.4, 2.7) circle (0.07cm);
  \draw[line width=0.901pt] (0.4, 3.3) -- (0.4, 2.7);
  \draw[line width=0.901pt] (3, 3) ellipse (1cm and 3cm);
  \filldraw (3, 3.3) circle (0.03cm);
  \filldraw (3, 3) circle (0.03cm);
  \filldraw (3, 2.7) circle (0.03cm);
  \draw (3.87, 4.5) -- (5.38, 4);
  \draw (3.87, 4.5) -- (5.38, 2);
  \draw (3.87, 1.5) -- (5.38, 4);
  \draw (3.87, 1.5) -- (5.38, 2);
  \draw[line width=0.901pt] (6, 3) ellipse (0.7cm and 2.1cm);
  \filldraw (6, 4) circle (0.07cm);
  \filldraw (6, 3.3) circle (0.03cm);
  \filldraw (6, 3) circle (0.03cm);
  \filldraw (6, 2.7) circle (0.03cm);
  \filldraw (6, 2) circle (0.07cm);
  \draw (2.25, 5)  -- (3.75, 5);
  \draw (2.08, 4.2)  -- (3.92, 4.2);
  \draw (2.00, 3.5)  -- (4.00, 3.5);
  \draw (2.00, 2.5)  -- (4.00, 2.5);
  \draw (2.08, 1.8)  -- (3.92, 1.8);
  \draw (2.25, 1)  -- (3.75, 1);
  \node[below] at (3,5.75) {\footnotesize $A_1$};
  \node[below] at (3,4.85) {\footnotesize $A_2$};
  \node[below] at (3.0,1.65) {\footnotesize $A_{2m+k-1}$};
  \node[below] at (3.0,1) {\footnotesize $A_{2m+k}$};
  \node[below] at (3,0) {\small $A$};
  \node[below] at (6,0.9) {\small $B$};
\end{tikzpicture}
\caption{The Hajnal graph $H(k, \ell, m)$.}\label{1.2}
\end{figure}

Lov\'asz \cite{lovasz1978kneser} famously proved that $\chi(\mathrm{KN}(n, m))=n-2m+2$. 
Clearly, $\chi(H(k, \ell, m))\ge\chi(\mathrm{KN}(2m+k,m))=k+2$ and $H(k, \ell, m)$ is triangle-free provided $m = \omega(k)$. Therefore, taking $\ell=\omega(\binom{2m+k}{m})$ and $k\ge C-2$, $G=H(k, \ell, m)$ is a $(3+o(1))\ell$-vertex triangle-free graph with $\delta(G)\ge\ell\ge(1/3-o(1))\cdot|G|$ and $\chi(G)\ge C$, verifying the lower bound $\delta_\chi(K_3)\ge 1/3$.

\begin{construction}[$r$-Hajnal graphs, extremal graphs for $\delta_{\chi}(K_r)$]\label{construction:K_r} Given $r\ge 3$, let $H(k, \ell, m)$ be the Hajnal graph defined in \cref{def:Hajnal-graph}. 
The $r$-Hajnal graph is defined to be $H_r(k, \ell, m):=H(k, \ell, m)\vee K_{r-3}[2\ell]$. 
\end{construction}

Since $H(k, \ell, m)$ is $K_3$-free, $H_r(k, \ell, m)$ is $K_r$-free. Setting $\ell = \omega(\binom{2m+k}{m})$ and $k\ge C-2$, it is not hard to see that $G=H_r(k, \ell, m)$ is a $(2r-3+o(1))\ell$-vertex $K_r$-free graph with $\delta(G)\ge(2r-5)\ell\ge(\frac{2r-5}{2r-3}-o(1))\cdot|G|$ and $\chi(G)\ge C$.

We remark that the \emph{$r$-Borsuk-Hajnal graph} differs from the $r$-Hajnal graph only at the Kneser graph part and the edges incident to it.

\paragraph{Extremal graphs for $\delta_\chi(H)=\frac{r-3}{r-2}$.} A slight modification of \cref{construction:extremal-graph-pi} shows that every graph $H$ with $\chi(H)=r\ge3$ satisfies $\delta_\chi(H)\ge\frac{r-3}{r-2}$ as follows.

\begin{construction}[Extremal graph for $\delta_\chi(H)=\frac{r-3}{r-2}$]\label{construction:extremal-graph-theta}
For every $H$ with $\chi(H)=r\ge3$ and every $C \in \mathbb{N}$, let $G'$ be a $(C,|H|+1)$-Erd\H{o}s graph. Let $G=G'\vee K_{r-3}[|G'|]$. 
\end{construction}

By construction, every induced subgraph of $G'$ with at most $|H|$ vertices is a forest and hence 2-colorable. Consequently, every induced subgraph of $G$ with at most $|H|$ vertices is $(r-1)$-colorable.  Therefore, $G$ is an $H$-free graph with $\delta(G)=\frac{r-3}{r-2}\cdot |G|$ and $\chi(G) \geq C$, thus proving the lower bound $\delta_\chi(H)\ge\frac{r-3}{r-2}$.   

\begin{remark}\label{remark}
  Note that if we replace the graph $G$ in \cref{construction:extremal-graph-theta} with $G^*:=(G' \cup K_{|G'|,|G'|})\vee K_{r-3}[3|G'|]$, then $G^*$ is also an $H$-free graph with $\delta(G^*)=\frac{r-3}{r-2} \cdot|G^*|$ and $\chi(G^*) \geq C$. 
  Here, $e(G^*)=e(T_{|G^*|, r-2})+ \Omega(|G^*|^2)$. More generally, we can replace the $(G' \cup K_{|G'|,|G'|})$ part with any graph in which any set of $|H|$ vertices induces a bipartite graph; such graph could have any edge density in $[0,1/2]$. Thus, we essentially have no control over the structure of the high chromatic number part of $G^*$.
\end{remark}

\subsection{Auxiliary results}\label{subsection2.2}
In this section, we present several auxiliary results that will be used throughout the paper. 

First we introduce the famous regularity lemma. Let $(A, B)$ be a pair of subsets of vertices of $G$. Let $e(A, B)$ denote the number of edges with one endpoint in $A$ and the other in $B$. Define the \emph{density} of the pair $(A, B)$ as $d(A, B)=\frac{e(A, B)}{|A||B|}$. For any $\varepsilon>0$, the pair $(A, B)$ is said to be \emph{$\varepsilon$-regular} if $|d(A, B)-d(X, Y)|<\varepsilon$ for every $X \subseteq A$ and $Y \subseteq B$ with $|X| \geq \varepsilon|A|$ and $|Y| \geq \varepsilon|B|$. Moreover, given $0 < d < 1$, we say that $(A, B)$ is \emph{$(\varepsilon, d)$-regular} if it is $\varepsilon$-regular and has density at least $d$.

A partition $V_0 \cup V_1 \cup \cdots \cup V_k$ of $V(G)$ is called an \emph{$\varepsilon$-regular partition} if $|V_0| \leq \varepsilon n,|V_1|=\cdots=|V_k|$, and all but at most $\varepsilon k^2$ of the pairs $(V_i, V_j)$, where $1\le i,j\le k,$ are $\varepsilon$-regular. Given an $\varepsilon$-regular partition $V_0 \cup V_1 \cup \cdots \cup V_k$ of $V(G)$ and $0<d<1$, we define a graph $R$, called the \emph{$(\varepsilon, d)$-reduced graph} of $G$, as follows: the vertex set is $V(R)=[k]$ and an edge $ij$ is included in $E(R)$ if and only if $(V_i, V_j)$ is an $(\varepsilon, d)$-regular pair. The partition classes $V_1, \ldots, V_k$ are called the \emph{clusters} of $G$. For brevity, for each $I\subseteq [k]$ we write $V_I=\cup_{i\in I}V_i$.

We will use the following minimum degree form of the regularity lemma.

\begin{theorem}[Theorem~1.10 of \cite{komlos1995szemeredi}]\label{thm:RL}
    Let $0<\varepsilon<d<\delta<1$, and let $k_0 \in \mathbb{N}$. There exists a constant $k_1=k_1(k_0, \varepsilon, \delta, d)$ such that the following holds. Every graph $G$ on $n>k_1$ vertices, with minimum degree $\delta(G) \geq \delta n$, has an $(\varepsilon, d)$-reduced graph $R$ on $k$ vertices, with $k_0 \leq k \leq k_1$ and $\delta(R) \geq(\delta-d-\varepsilon) k$.
\end{theorem}

We will also make use of the counting lemma.

\begin{theorem}[Counting lemma, Theorem~3.1~\cite{komlos1995szemeredi}]\label{thm:counting}
Let $G$ be a graph with $(\varepsilon, d)$-reduced graph $R$ whose clusters contain $m$ vertices each, and suppose that there is a homomorphism $\phi: V(H) \rightarrow V(R)$. Then $G$ contains at least
\begin{align*}
\frac{1}{|\operatorname{Aut}(H)|}(d-\varepsilon|H|)^{e(H)} m^{|H|}
\end{align*}
copies of $H$, each with the property that every vertex $x \in V(H)$ lies in the cluster corresponding to the vertex $\phi(x)$ of $R$.  
\end{theorem}

We state one more useful fact about subpairs of $(\varepsilon, d)$-regular pairs.

\begin{fact}[Slicing lemma, see e.g. Fact 1.5 in \cite{komlos1995szemeredi}]\label{fact:slicinglemma} 
  Let $\varepsilon<\alpha\le 1/2$.
    Let $(U, W)$ be an $(\varepsilon, d)$-regular pair and suppose that $U' \subseteq U$, $W' \subseteq W$ satisfy $|U'| \geq \alpha|U|$ and $|W'| \geq \alpha|W|$. Then $(U', W')$ is $(\varepsilon / \alpha, d-\varepsilon)$-regular.
\end{fact}

The following two lemmas are basic facts of reduced graphs.

\begin{lemma}\label{prop:16}
  Let $G$ be a graph with an $(\varepsilon, d)$-reduced graph $R$ on $k$ vertices, where $1/k\le 2\varepsilon$. 
  Let $I$ be an independent set of $R$ and let $V_I=\bigcup_{i\in I}V_i$. Then $e(G[V_I])<(2\varepsilon+d/2)n^2.$
\end{lemma}
\begin{proof}
 Since $I$ is an independent set in $R$, the edges of $G[V_I]$ can be categorized into the following three types:
\begin{itemize}
\item  edges with both endpoints in $V_i$ for some $i\in I$. There are at most $k\binom{n/k}{2}$ such edges;
\item edges between distinct clusters $V_i$ and $V_j$, for some distinct $i,j\in I$ for which $(V_i,V_j)$ is not $\varepsilon$-regular.  There are at most $\varepsilon k^2\frac{n^2}{k^2}$ such edges;
\item edges between distinct clusters $V_i$ and $V_j$, for some distinct $i,j\in I$ for which $(V_i,V_j)$ is $\varepsilon$-regular but $d(V_i,V_j)<d$. There are at most $d\binom{k}{2}\frac{n^2}{k^2}$ such edges.
\end{itemize} 
Summing up the edges from these three types yields the desired bound.   
\end{proof}

\begin{lemma}\label{lmm:edit}
Let $r\geq 3$ be an integer and $\varepsilon,\beta,\mu,d>0$ with $\varepsilon\le d\le \frac{\beta}{20r}$. Let $G$ be an $n$-vertex graph with minimum degree $\delta(G)\ge(\frac{r-1}{r}-\beta)n$.  Suppose that $V(G)=V_0\cup V_1\cup\cdots\cup V_k$ is the $\varepsilon$-regular partition of $G$ where $1/k\le 2\varepsilon$ and $R$ is the $(\varepsilon,d)$-reduced graph of $G$ from this partition. If $|E(R)\triangle E(T_{k,r})|\le\mu k^2$, then $|E(G)\triangle E(T_{n,r})|\le(\beta+2\mu)n^2.$
\end{lemma}

\begin{proof}
 Since $|E(R)\triangle E(T_{k,r})|\le\mu k^2$, there exists a balanced $r$-partition of $V(R)$, say $[k]=U_1\cup\cdots\cup U_r$, such that $\sum_{i=1}^r e(R[U_i])\leq \mu k^2$. For each $i\in [r]$, let $W_i=\cup_{j\in U_i} V_j$. By \cref{prop:16}, 
\begin{align*}
\sum_{i=1}^r e(G[W_i])\leq r(2\varepsilon+d/2)n^2+\mu k^2\frac{n^2}{k^2}<(\beta/8+\mu)n^2.
\end{align*}

Since $e(G)\ge\frac{\delta(G) n}{2}$ and the vertices in $V_0$ are incident with at most $\varepsilon n^2$ edges, the $r$-partite subgraph of $G$ defined on the vertex set $W_1\cup\cdots\cup W_r$ satisfies
\begin{align*}
e(G[W_1,\ldots,W_r])
&\geq \frac{\delta(G)n}{2}-\varepsilon n^2-\sum_{i=1}^r e(G[W_i])\geq \Big(\frac{r-1}{2r}-\frac{\beta}{2}\Big)n^2-\varepsilon n^2-\Big(\frac{\beta}{8}+\mu\Big)n^2\\
&=\Big(\frac{r-1}{2r}-\frac{5\beta}{8}-\mu-\varepsilon\Big)n^2.
\end{align*}

Since $e(T_{n,r})\le\frac{r-1}{2r}\cdot n^2$ and $\varepsilon\le\frac{\beta}{60}$, we have 
\begin{align*}
|E(G)\triangle E(T_{n,r})|&=|E(G)\setminus E(T_{n,r})|+|E(T_{n,r}) \setminus E(G)| \\
&\le \Big(\varepsilon+\frac{\beta}{8}+\mu\Big)n^2+\Big(\frac{5\beta}{8}+\mu+\varepsilon\Big)n^2\le(\beta+2\mu)n^2,
\end{align*}
as needed.
\end{proof}

The following two lemmas from \cite{allen2013chromatic} describe some useful properties of graphs. 

\begin{lemma}[\cite{allen2013chromatic}, Lemma 9]\label{lmm:lmm9}
Let $\alpha, \delta>0$ and $r, t \in \mathbb{N}$. Let $G$ be a graph on $n$ vertices, and $X \subseteq V(G)$.
\begin{enumerate}
    \item[\rm (a)] If $\delta(G) \geq \delta n$ and $ |X| \geq\big(\alpha^\frac{1}{r-2} (r-2)+(1-\delta)(r-3)\big) n$, then $G[X]$ contains at least $\alpha n^{r-2}$ copies of $K_{r-2}$.
    \item[\rm (b)] For any graph $H$ with $\chi(H)\le r$, if $G[N(x)]$ contains at least $\alpha n^{(r-1)|H|}$ copies of $K_{r-1}[|H|]$ for every $x \in X$, then we have either $H \subseteq G$ or $|X| \leq |H| / \alpha$.
\end{enumerate}
\end{lemma}

Although Lemma 9 in \cite{allen2013chromatic} is stated for graphs $H$ satisfying $H \subseteq F \vee K_{r-2}[t]$ for some foreset $F$, this condition is not used in the proof. 
Hence the above lemma also holds.

\begin{lemma}[\cite{allen2013chromatic}, Lemma 10]\label{lmm:lmm10}
    Let $\alpha, \delta>0$ and $r, t \in \mathbb{N}$, let $F$ be a forest, and suppose that $H \subseteq F\vee K_{r-2}[t]$. Let $G$ be an $H$-free graph on $n$ vertices, and let $X \subseteq V(G)$ be such that every edge $x y \in E(G[X])$ is contained in at least $\alpha n^{r-2}$ copies of $K_{r}$ in $G$.
    Then 
    $\chi(G[X]) \leq(2|F| / \alpha')+1$ for some $\alpha'=\alpha'(\alpha,r,t)$.
\end{lemma}

The following two well-known facts are useful.

\begin{fact}\label{fact:4}
  Let $F$ be a forest and $G$ be a graph on $n \geq 1$ vertices. 
  If $e(G) \geq|F| n$, then $F \subseteq G$.
\end{fact}

\begin{fact}\label{fact:forest-free}
Let $F$ be a forest and $G$ be an $F$-free graph, then $\chi(G)\le |F|$.
\end{fact}

We will utilize the classic result of Andr\'asfai, Erd\H{o}s, and S\'os \cite{andrasfai1974connection}. 

\begin{theorem}[\cite{andrasfai1974connection}]\label{thm:1974Andrasfai}
Let $r\geq 3$ and let $G$ be a $K_r$-free graph on $n$ vertices such that $\delta(G)>\frac{3 r-7}{3 r-4}\cdot n$. Then $\chi(G) \leq r-1$.
\end{theorem}

\section{Stability for $\delta_{\chi}(H)=\frac{2r-5}{2r-3}$}\label{sec:lambda-stability}
In this section we prove the following theorem, which implies \cref{thm:lambda-stability}.

\begin{theorem}\label{thm:lambda-stability-precise-version} 
Let $r\ge 3$ and $H$ be a graph with $\chi(H)=r$ and $\delta_\chi(H)=\frac{2r-5}{2r-3}$.
For every $0<\beta<(40r|H|)^{-2}$, there exist $C=C(H,\beta)$ and $C'=C'(H)$ such that the following hold.

For every $n$-vertex $H$-free graph $G$ with $\delta(G)\ge\big(\frac{2r-5}{2r-3}-\beta\big)n$ and $\chi(G)\ge C$, we have
\begin{align}\label{eq:poker}
\Big|E(G)\triangle E\Big(K_{r-1}\Big(\frac{2n}{2r-3},\ldots,\frac{2n}{2r-3},\frac{n}{2r-3}\Big)\Big)\Big|\le 20r|H|\sqrt{\beta} n^2. 
\end{align}
Moreover, $G$ admits a vertex partition $V(G)=A\cup B_1\cup\cdots\cup B_{r-1}$ with the following properties:
\begin{itemize}
\item [{\rm (i)}] $|A|\le 7r|H|\sqrt{\beta}n$, $|B_i|=\big(\frac{2}{2r-3}\pm 2\sqrt{\beta}\big)n$ for each $i\in[r-2]$, and $|B_{r-1}|=\big(\frac{1}{2r-3}\pm  5r|H|\sqrt{\beta}\big)n$. Moreover, for every $i\in [r-2]$ and $v\in B_i$, $v$ has at most $5\sqrt{\beta}n$ non-neighbors outside $B_i$;
\item [{\rm (ii)}] for every forest $F\in\mathcal{M}(H)$ and $i\in[r-1]$, $G[B_i]$ is $F$-free;
\item [{\rm (iii)}]  $\chi(G[A])\ge C-(r-1)|H|$ but $\chi_f(G[A])\le C'$.
\end{itemize} 
\end{theorem}


\begin{proof}
Let $C$ be a sufficiently large integer whose value will be determined later.
Let $G$ be an $H$-free graph on $n$ vertices with $\delta(G)\ge\big(\frac{2r-5}{2r-3}-\beta\big)n$ and $\chi(G)>C$.\footnote{As $C$ is sufficiently large, we can assume $n\ge\chi(G)>C$ is also sufficiently large.} 
Applying \cref{thm:RL} to $G$ with $d:=\frac{\beta}{4}$, $\varepsilon:=\frac{d^2}{2d+2|H|}$ and $k_0=\lceil\frac{1}{2\varepsilon}\rceil$, we obtain a partition $V(G)=V_0\cup V_1\cup \cdots \cup V_k$ of $G$, where $k_0\leq k\leq k_1(k_0,\varepsilon,\delta,d)=k_1(H,\beta)$, together with an $(\varepsilon,d)$-reduced graph $R$ on the vertex set $V(R)=[k]$, such that
$\delta(R)\geq\big(\frac{2r-5}{2r-3}-2\beta\big)k$.  

Define a new partition of $V(G)$ by setting, for every subset $I\subseteq [k]$,
\begin{align*}
X_I:=\{x\in V(G): i\in I\Leftrightarrow|N(x)\cap V_i|\geq d|V_i|\}.
\end{align*}
That is, $X_I$ consists of all vertices having linear degree to $V_i$ when and only when $i\in I$.

First, we show that if $|I|$ is not large, then $\chi(G[X_I])$ is necessarily bounded.
\begin{claim}\label{claim:chi(X_I)-bounded-for-small_I}
There exists $C_1=C_1(H,\beta)$ such that for every $I\subseteq V(R)$ with $|I|\leq (\frac{2r-4}{2r-3}-(2r-1)\beta )k$, we have $\chi(G[X_I])\leq C_1$.
\end{claim}

\begin{pf}
For any edge $xy\in E(G[X_I])$, let $X= N(x)\cap N(y)\cap V_I$. Note that
$$\min\{|N(x)\cap V_I|,|N(y)\cap V_I|\}\ge \delta(G)-|V_0|-d(n-|V_0|)\ge \Big(\frac{2r-5}{2r-3}-2\beta\Big)n,$$
which implies that
\begin{align*}
|X|\ge 2\cdot\Big(\frac{2r-5}{2r-3}-2\beta\Big)n-|V_I|\ge \Big(\frac{2r-6}{2r-3}+(2r-5)\beta\Big)n.
\end{align*}
Applying \cref{lmm:lmm9}(a) with $\alpha=\beta^{r-2}, \delta=\frac{2 r-5}{2 r-3}-\beta$, and $X=N(\{x, y\})$, $G[N(\{x,y\})]$ contains at least $\beta^{r-2}n^{r-2}$ copies of $K_{r-2}$. 
Since $G$ is $H$-free,  \cref{lmm:lmm10} shows that $\chi(G[X_I])$  is bounded by some constant $C_1=C_1(H,\beta)$.
\end{pf}

Next we show that if $R[I]$ contains a $K_{r-1}$, then $\chi(G[X_I])$ is also bounded.
\begin{claim}\label{claim:chi(X_I)-bounded-for-large_I}
    There exists $C_2=C_2(H,\beta)$ such that for every $I\subseteq V(R)$ with $K_{r-1}\subseteq R[I]$, we have $\chi(G[X_I])\leq C_2$.
\end{claim}  
\begin{pf}
  Suppose that $R[I]$ contains a copy of $K_{r-1}$ with vertex set $[r-1] \subseteq I$.
  For any $x \in X_I$, we have $|N(x) \cap V_i| \ge d |V_i|$ for all $i \in [r-1]$.
 By \cref{fact:slicinglemma}, it follows that for every pair $\{i, j\} \in \binom{[r-1]}{2}$, the pair $(N(x) \cap V_i, N(x) \cap V_j)$ is $(\varepsilon/d, d - \varepsilon)$-regular. Hence, by \cref{thm:counting}, $G[N(x)]$ contains at least $\alpha n^{(r-1)|H|}$ copies of $K_{r-1}[|H|]$ for some $\alpha=\alpha(H,\beta)>0$. 
  Since $G$ is $H$-free, it follows from \cref{lmm:lmm9}(b) that $\chi(G[X_I])\leq |X_I|\leq |H| / \alpha$.
  By letting $C_2(H,\beta) = |H|/\alpha$, this proves the claim.
\end{pf}

Let $C= 3^k \max\{C_1, C_2\}$.
As $\sum_{I\subseteq[k]}\chi(G[X_I])\ge \chi(G)\ge C$, it follows from the pigeonhole principle that there exists a set $I\subseteq [k]$ such that $\chi(G[X_I])\ge C/2^k > \max\{C_1,C_2\}$. 
Fix such a set $I$.
Then by \cref{claim:chi(X_I)-bounded-for-small_I} and \cref{claim:chi(X_I)-bounded-for-large_I}, $|I|\ge\big(\frac{2r-4}{2r-3}-(2r-1)\beta\big)k$ and $R[I]$ is $K_{r-1}$-free.
Moreover, for $r\geq 4$, 
\begin{align*}
\delta(R[I])&\geq \delta(R)-(k-|I|)\ge|I|-\Big(\frac{2}{2r-3}+2\beta\Big)k>\frac{3r-10}{3r-7}|I|.
\end{align*}
It follows from \cref{thm:1974Andrasfai} that $\chi(R[I])\le r-2$, and this holds for $r=3$ as well since $R[I]$ is an independent set in this case.

\vskip 0.25em 
Now, we describe the process of constructing the partition.  
Recall that any vertex in $R$ has at most $(\frac{2}{2r-3}+2\beta)k$ non-neighbors. Consequently, any independent set in $R$ has size at most $(\frac{2}{2r-3}+2\beta)k$.

\begin{itemize}
\item \textbf{Step 1.} As $\chi(R[I])\le r-2$, $I$ admits a partition $I=\bigcup_{i=1}^{r-2}Z_i$, where each $Z_i$ is an independent set in $R$.
\item \textbf{Step 2.} For each $i\in[r-2]$, let $V_{Z_i}=\bigcup_{j\in Z_i} V_j$, $B_i=\{v\in V_{Z_i}: |N(v)\cap V_{Z_i}|< 2\sqrt{\beta}n\}\subseteq V_{Z_i}$.
\item \textbf{Step 3.} Select sets $T_1,\dots, T_{r-2}$ in order,  each of size $|H|$, such that 
\[T_1\subseteq B_1,~T_2\subseteq B_2\cap N(T_1),\cdots,T_{r-2}\subseteq B_{r-2}\cap N\left(\bigcup_{i=1}^{r-3} T_i\right).\]
\item  \textbf{Step 4.}   Let \[B_{r-1}=N\left(\bigcup_{i=1}^{r-2} T_i\right) \setminus \bigcup_{i=1}^{r-2} B_i, \quad \quad A=V(G)\setminus \left(\bigcup_{i=1}^{r-1} B_i\right).\]
\end{itemize}

We first prove that this process can be executed. To demonstrate this, it suffices to show that we can select $T_1,\cdots,T_{r-2}$ in Step 3. To this end, we need to bound the size of each $B_i$.
\begin{claim}\label{cl:3.3}
For each $i\in [r-2]$, $\big(\frac{2}{2r-3}-2\sqrt{\beta}\big)n< |B_i|\leq \big(\frac{2}{2r-3}+2\beta\big)n.$
\end{claim}
\begin{pf}    
Since $Z_i$ is an independent set with $|Z_i|\le(\frac{2}{2r-3}+2\beta)k$, we have
\begin{align*}
|Z_i|\ge|I|-(r-3)\Big(\frac{2}{2r-3}+2\beta\Big)k>\Big(\frac{2}{2r-3}-4r\beta\Big)k, 
\end{align*}
which concludes that $(\frac{2}{2r-3}-\sqrt{\beta})n\le|V_{Z_i}|\le(\frac{2}{2r-3}+2\beta)n$ where the lower bound follows from $\beta<\frac{1}{25r^2}$. Moreover, by \cref{prop:16} we have $e(G[V_{Z_i}])<\big(2\varepsilon+d/2\big)n^2<\beta n^2$. Hence, 
\begin{align*}
\Big(\frac{2}{2r-3}-2\sqrt{\beta}\Big)n< |V_{Z_i}|-\frac{2e(G[V_{Z_i}])}{2\sqrt{\beta}n}\le|B_i|\leq|V_{Z_i}|\leq \Big(\frac{2}{2r-3}+2\beta\Big)n,
\end{align*}
as claimed.
\end{pf}

The next claim shows that each vertex in $B_i$ has only a few non-neighbors outside $B_i$.
\begin{claim}\label{claim:few-non-neighbors-Kr}
For every $i\in [r-2]$ and $v\in B_i$, $v$ has at most $5\sqrt{\beta}n$ non-neighbors outside $B_i$, that is, $|(V(G)\setminus B_i)\setminus N(v)|\le 5\sqrt{\beta}n$.
\end{claim}

\begin{pf}
As $\delta(G)\ge\big(\frac{2r-5}{2r-3}-\beta\big)n$, $v$ is non-adjacent to at most $\big(\frac{2}{2r-3}+\beta\big)n$ vertices in $V(G)$. 
Moreover, by the definition of $B_i$, we have $|N(v)\cap B_i|\le|N(v)\cap V_{Z_i}|<2\sqrt{\beta}n$. Thus, by \cref{cl:3.3}, $v$ is non-adjacent to at least $|B_i|-2\sqrt{\beta}n\ge\big(\frac{2}{2r-3}-4\sqrt{\beta}\big)n$ vertices in $B_i$. It follows that $v$ is non-adjacent to at most $(4\sqrt{\beta}+\beta)<5\sqrt{\beta}n$ vertices in $V(G)\setminus B_i$.
\end{pf}

Now, we can show that the process in Step 3 is well-defined. Assume that we have chosen $T_1,\ldots, T_{i-1}$. Then, by \cref{cl:3.3} and \cref{claim:few-non-neighbors-Kr}, we obtain that 
\begin{align*}
\left|B_i\cap N\left(\bigcup_{j=1}^{i-1} T_j\right)\right|
\ge |B_i|-5(i-1)|H|\sqrt{\beta}n
 \ge\Big(\frac{2}{2r-3}-2\sqrt{\beta}\Big)n-5(r-3)|H|\sqrt{\beta}n>|H|.
\end{align*}
Thus we can choose the desired $T_{i}$ in $B_i$ for each $i\in[r-2]$ in Step 3.

\smallskip

Next we show that the partition $V(G)=A\cup B_1\cup\ldots\cup B_{r-1}$ satisfies the properties listed in \cref{thm:lambda-stability-precise-version}. Let $B=B_1\cup \cdots \cup B_{r-2}$. By \cref{cl:3.3}, we have $(\frac{2r-4}{2r-3}-2r\sqrt{\beta}) n\le|B|\le(\frac{2r-4}{2r-3}+2r\beta) n.$ Thus, it follows from~\cref{claim:few-non-neighbors-Kr} that
\begin{align*}
\Big(\frac{1}{2r-3}+2r\sqrt{\beta}\Big) n \ge n-|B|\ge |B_{r-1}|
&= \left|N\left(\bigcup_{i=1}^{r-2} T_i\right)\cap(V(G)\setminus B)\right|\\
&\ge|V(G)\setminus B|-(r-2)|H|5\sqrt{\beta}n\ge \Big(\frac{1}{2r-3}-5r|H|\sqrt{\beta}\Big)n.
\end{align*}
Finally, since $|A|=n-|B|-|B_{r-1}|\le 7r|H|\sqrt{\beta}n$,  we complete the proof of \cref{thm:lambda-stability-precise-version}(i).

\vskip 0.25em

Now, we prove \cref{thm:lambda-stability-precise-version}(ii). By our construction, it is clear that $G[T_1,\ldots,T_{r-2},B_{r-1}]$ is a complete $(r-1)$-partite graph and  $G[T_1,\ldots,T_{r-2}]=K_{r-2}[|H|]$. Hence, $B_{r-1}$ is $F$-free for every $F\in \mathcal{M}(H)$; otherwise, $G$ contains a copy of $F\vee K_{r-2}[|H|]\supseteq H$, a contradiction.  
Suppose that $G[B_1]$ contains a copy of $F\in \mathcal{M}(H)$. Applying the same process with $T_1$ containing the vertex set of this copy of $F$, it follows that $N(T_1)$ also contains a copy of $K_{r-2}[|H|]$, which is again a contradiction. Similarly, one can show that $G[B_i]$ is $F$-free for each $i\in[r-2]$, as desired.

\vskip 0.25em

We now proceed to prove \cref{thm:lambda-stability-precise-version} (iii).  
By \cref{fact:forest-free},  $\chi (G[B\cup B_{r-1}])\le (r-1)|F|\le (r-1)|H|$. 
Thus, $\chi(G[A])\ge C-(r-1)|H|$. Moreover, noting that $\delta(G)>(\frac{2r-5}{2r-3}-\beta)n>(\frac{r-3}{r-2}+\frac{1}{r^3}) n$, it follows from \cref{thm:fractional-r} (applied to $G$) that $\chi_f(G[A])\le \chi_f(G)\le C'=C'(H)$, which verifies \cref{thm:lambda-stability-precise-version}(iii).

\vskip 0.25em
Finally, we prove \eqref{eq:poker}.
Since for every $i\in [r-2]$ and $v\in B_i$, $v$ is non-adjacent to at most $5\sqrt{\beta}n$ vertices outside $B_i$, the total number of missing edges among all distinct pairs of $B_i,B_j$ is at most $\frac{5\sqrt{\beta}n^2}{2}<3\sqrt{\beta}n^2$. Moreover, since for every forest $F\in\mathcal{M}(H)$ and every $i\in[r-1]$, $G[B_i]$ is $F$-free, by \cref{fact:4}, the number of edges within the $B_i$'s is at most $(r-1)|H|n<\sqrt{\beta}n^2.$ Hence,
\begin{align*}
&\Big|E(G)\triangle E\Big(K_{r-1}\Big(\frac{n}{2r-3},\frac{2n}{2r-3},\ldots,\frac{2n}{2r-3}\Big)\Big)\Big|\\
\le& |A|n+\sqrt{\beta}n^2+3\sqrt{\beta}n^2+n\cdot\Big(\sum_{i=1}^{r-2}\Big||B_i|-\frac{2n}{2r-3}\Big|+\Big||B_{r-1}|-\frac{n}{2r-3}\Big|\Big)
\le 20r|H|\sqrt{\beta} n^2,
\end{align*}
as needed.
\end{proof}

\section{Stability for $\delta_\chi(K_r)$}\label{sec:Kr-stability}
In this section we prove the following theorem, which implies \cref{thm:Kr}. 

\begin{theorem}\label{thm:Kr-stability-precise}
For every $r \geq 3$ and $0<\beta<40^{-2}r^{-4}$, there exists $C=C(r, \beta)$  such that the following hold. 
For every $n$-vertex $K_r$-free graph $G$ with $\delta(G)\ge\big(\frac{2r-5}{2r-3}-\beta\big)n$ and $\chi(G)\ge C$, $G$ admits a vertex partition $V(G)=$ $A^* \cup B_1^* \cup \cdots \cup B_{r-1}^*$ with the following properties:

\begin{itemize}
\item [{\rm (i)}] $|A^*|\le 7r^2\sqrt{\beta}n,|B_i^*|=\big(\frac{2}{2 r-3} \pm 2\sqrt{\beta})\big) n$ for each $i \in[r-2]$, and $|B_{r-1}^*|=\big(\frac{1}{2 r-3} \pm 5r^2\sqrt{\beta}\big) n$;
\item [{\rm (ii)}]  for each $i \in[r-1], B_i^*$ is an independent set;
\item [{\rm (iii)}]  $\chi(G[A^*]) \geq C-r+1$ and $\chi_f(G[A^*]) \leq 2+50r^3\sqrt{\beta}$; in particular, $G[A^*]$ is homomorphic to $\mathrm{KN}(m, (\frac{1}{2}-\varepsilon)m)$, where $m=\ln n/\beta$ and $\varepsilon=100r^3\sqrt{\beta}$.
\end{itemize} 
\end{theorem}

\begin{proof}
Let $C$ be $C(K_r, \beta)$ given from \cref{thm:lambda-stability-precise-version}. Suppose that $G$ is a $K_r$-free graph with $\delta(G)\ge (\frac{2r-5}{2r-3}-\beta)n$ and $\chi(G) \geq C$. Applying \cref{thm:lambda-stability-precise-version}, let $V(G)=A\cup B_1\cup \ldots\cup B_{r-1}$ be a partition satisfying the conclusion of \cref{thm:lambda-stability-precise-version}. 

We first note that $G[B_1,\ldots,B_{r-1}]$ is rich in $K_{r-1}$.

\begin{claim}\label{fact:K_{r-1}}
Let $T_1\subseteq B_1,\cdots,T_{r-1}\subseteq B_{r-1}$ be subsets with $|T_i|\ge 5r\sqrt{\beta}n$ for each $i\in [r-1]$.  Then $K_{r-1}\subseteq G[T_1,\ldots,T_{r-1}]$.
\end{claim}

\begin{pf}
We can greedily pick $w_1\in T_1, w_2\in N(w_1)\cap T_2,\ldots, w_{r-1}\in N(\{w_1,\ldots,w_{r-2}\})\cap T_{r-1}$,
so that $\{w_1,\ldots,w_{r-1}\}$ forms a copy of $K_{r-1}$. 
To see this, for every $i\in [r-2]$,  by \cref{thm:lambda-stability-precise-version}(i), we have $|T_{i+1}\setminus N(w_j)|\le 5\sqrt{\beta}n$ for each $j\in [i]$. Hence
\begin{align*}
|N(\{w_1,\ldots,w_{i}\})\cap T_{i+1}|\ge|T_{i+1}|-\sum_{j=1}^{i}|T_{i+1}\setminus N(w_j)|\ge 5r\sqrt{\beta}n-5i\sqrt{\beta}n>0.
\end{align*}
Therefore, the vertices $w_1,\ldots,w_{r-1}$ are well-defined.
\end{pf}
In order to obtain the desired sets $B^*_i$ in \cref{thm:Kr-stability-precise}, we need a refined analysis on $G[A]$.

\begin{claim}\label{cl:26}
For every vertex $x\in A$, there exists exactly one $i\in [r-1]$ such that $|N(x)\cap B_i|< 5r\sqrt{\beta}n$.
\end{claim}

\begin{pf}
  Fix an arbitrary $x\in A$. 
  Since $|N(x)|\ge\big(\frac{2r-5}{2r-3}-\beta\big)n$, $|B_{r-1}|\ge (\frac{1}{2r-3}-5r^2\sqrt{\beta}) n$ and $|B_i|>(\frac{2}{2r-3}-2\sqrt{\beta})n$ for $i\in[r-2]$, there is at most one $i\in [r-1]$ such that $|N(x)\cap B_i|< 5r\sqrt{\beta}n$. 
  On the other hand, if $|N(x)\cap B_i|\geq 5r\sqrt{\beta}n$ for every $i\in[r-1]$, 
  then applying \cref{fact:K_{r-1}} with $T_i:=N(x)\cap B_i$, it follows that $K_{r-1}\subseteq G[T_1,\ldots,T_{r-1}]\subseteq G[N(x)]$, which implies that $K_r\subseteq G$,  a contradiction.
\end{pf}

  For each $i\in[r-1]$, let 
  $$R_i=\{x\in A:|N(x)\cap B_i|< 5r\sqrt{\beta}n\}.$$ 
  Then by~\cref{cl:26}, $R_1,\ldots, R_{r-1}$ are pairwise disjoint and form a partition of $A$.
  Set 
  $$B^*_{r-1}:=B_{r-1} \quad \text{ $B^*_i:=B_i\cup R_i$ for $i\in[r-2]$, and $A^*=V(G)\setminus \left(\bigcup_{i=1}^{r-1} B^*_i\right)=R_{r-1}$.}$$ 
  Next we demonstrate that these $B^*_i$'s satisfy the conclusions (i)-(iii) of \cref{thm:Kr-stability-precise}. 
  
  We first prove part (ii). Notice first that as $\mathcal{M}(K_r)=\{K_2\}$, it follows immediately from \cref{thm:lambda-stability-precise-version}(ii) that each $B_i$, $i\in[r-1]$, is an independent set. 
  As $B^*_{r-1}:=B_{r-1}$ and $B^*_i$, $i\in[r-2]$, are symmetric, it suffices to show that $B^*_{1}$ is independent.
  By the definition of $R_1$ and that $B_1$ is independent, every vertex $v\in B^*_1$ satisfies 
  $$|N(v)\cap B^*_1|\le |R_1|+|N(v)\cap B_1|\le |A|+5r\sqrt{\beta}n\le 10r^2\sqrt{\beta}n.$$ 
  As $A\cup B_1\cup \dots \cup B_{r-1}$ satisfies \cref{thm:lambda-stability-precise-version}(i) and $|B^*_1|>(\frac{2}{2 r-3}-2 \sqrt{\beta}) n$, $v$ is non-adjacent to at least $(\frac{2}{2 r-3}-12r^2 \sqrt{\beta}) n$ vertices in $B^*_1$. Moreover, since $v$ is non-adjacent to at most $(\frac{2}{2 r-3}+\beta) n$ vertices in $V(G)$, we see that $v$ is non-adjacent to at most $13r^2 \sqrt{\beta} n$ vertices in $V(G) \backslash B^*_1$. 
  To show that $B^*_{1}$ is independent, suppose to the contrary that $xy\in E(G)$ for some distinct $x,y\in B^*_1$.
  Let $T_j:=N(x)\cap N(y)\cap B^*_j$ for $j\neq 1$. Thus, we have $|T_j|\ge |B^*_j|-26r^2\sqrt{\beta}n\ge5r\sqrt{\beta}n$ for each $j\neq 1$. Applying \cref{fact:K_{r-1}} with $T_j$ for $2\le j\le r-1$ gives that $K_{r-2}\subseteq G[T_2,\ldots,T_{r-1}]\subseteq G[N(x)\cap N(y)]$, a contradiction. This proves part (ii).
  
Now, part (i) is obvious since $|A^*|\le |A|\le 7r^2\sqrt{\beta}n$, $|B^*_{r-1}|=|B_{r-1}|$ and by part (ii), for $i\in[r-2]$, we have $|B_i|\le|B^*_i|\le n-\delta(G)\le \left(\frac{2}{2r-3}+\beta\right)n$ .

\vskip 0.25em
Finally, we  prove part (iii).  It is clear that $\chi(G-A^*)=r-1$ and $\chi(G[A^*])\ge C-r+1$.
\begin{claim}\label{lem:A-structure}
For any edge $xy\in E(G[A^*])$, there exists a unique $i\in [r-2]$ such that $N(\{x,y\})\cap B_i^*=\varnothing$.
\end{claim}

\begin{pf}
Let $B^*:=\bigcup_{i=1}^{r-2} B^*_i$, then $|B^*|\le n-|B^*_{r-1}|\le(\frac{2r-4}{2r-3}+5r^2\sqrt{\beta})n.$ For each $v\in A^*=R_{r-1}$, we have 
\begin{align*}
|N(v)\cap B^*|\ge \Big(\frac{2r-5}{2r-3}-\beta\Big)n-|A^*|-|N(v)\cap B^*_{r-1}|\ge \Big(\frac{2r-5}{2r-3}-13r^2\sqrt{\beta}\Big)n,
\end{align*}
Fix an edge $xy\in E(G[A^*])$, we then get 
\begin{align}\label{eq:N(xy)-cap-B}
|N(\{x,y\})\cap B^*|\ge|N(x)\cap B^*|+|N(y)\cap B^*|-|B^*|\ge\Big(\frac{2r-6}{2r-3}-31r^2\sqrt{\beta}\Big)n.
\end{align}

Notice that there is a unique $i\in [r-2]$ such that $|N(\{x,y\})\cap B^*_i|<5r\sqrt{\beta}n$. Indeed, if for each $i\in [r-2]$, $|N(\{x,y\})\cap B^*_i|\ge 5r\sqrt{\beta}n$, then we can apply~\cref{fact:K_{r-1}} with $T_j=N(\{x,y\})\cap B^*_j$ for $j\in [r-2]$ to obtain a copy of $K_{r-2}$ in $G[N(x)\cap N(y)]$, a contradiction. But if there are at least two such indices, then as $|B_j^*|\ge (\frac{2}{2r-3}-2\sqrt{\beta})n$, $x$ and $y$ would have too few common neighbors in $B^*$, contradicting~\eqref{eq:N(xy)-cap-B}. 

We shall show that this unique index $i\in [r-2]$ is the desired one. Suppose for contradiction that there exists a vertex $w\in B^*_i$ such that $w\in N(\{x,y\})$. 
Since for each $j\in[r-2]\setminus\{i\}$, $|B^*_j|\le |B_j|+|A|\le \big(\frac{2}{2r-3}+9r^2\sqrt{\beta}\big)n$, we have
\begin{align*}
|N(\{x,y\})\cap B^*_j|\ge|N(\{x,y\})\cap B^*|-|N(\{x,y\})\cap B^*_i|-\sum_{\ell\in[r-2]\setminus\{i,j\}}|B^*_\ell|\ge\Big(\frac{2}{2r-3}-45r^3\sqrt{\beta}\Big)n.
\end{align*}
Since $w\in B_i^*$,  by \cref{thm:lambda-stability-precise-version} (i), for each $j\in[r-2]\setminus\{i\}$, $|N(w)\cap B^*_j|\ge|B^*_j|-5\sqrt{\beta}n\ge\big(\frac{2}{2r-3}-7\sqrt{\beta}\big)n.$
Therefore, $|N(\{x,y,w\})\cap B^*_j|\ge 5r\sqrt{\beta}n$ for each $j\in[r-2]\setminus\{i\}$, and by \cref{fact:K_{r-1}} we have $K_{r-3}\subseteq G[N(\{x,y,w\})]$, a contradiction.
\end{pf}

Let $a:=(\frac{2}{2r-3}+\beta)n$ and $b:=(\frac{1}{2r-3}-12r^2\sqrt{\beta})n$. We aim to show that $G[A^*]\xrightarrow{\textup{hom}}\mathrm{KN}(a,b)$. It suffices to construct a map $\phi:A'\rightarrow\binom{[a]}{b}$ for each  connected component $A'\subseteq A^*$ such that if $uv\in E(G[A'])$, then $\phi(u)\cap \phi(v)=\varnothing$. Fix a connected component $A'$ in $A^*$. We next observe that each vertex in $A'$ is adjacent to at least almost half of each $B_j^*$, $j\in[r-2]$.

\begin{claim}\label{cl:A-deg}
    For any $v\in A'$ and any $j\in [r-2]$, $|N(v)\cap B^*_j|
      \ge \big(\frac{1}{2r-3}-12r^2\sqrt{\beta}\big)n.$
\end{claim}
\begin{pf}
     Note that 
  \begin{align*}
      |N(v)\cap B^*_j|\ge |N(v)|-|A^*|-|N(v)\cap B^*_{r-1}|-\sum_{\ell\in[r-2]\setminus\{j\}} |B^*_{\ell}|.
  \end{align*}
Recall that $B^*_{r-1}=B_{r-1}$ and $A^*=R_{r-1}$, and so $|N(v)\cap B^*_{r-1}|\le 5r\sqrt{\beta}n$.
Since $B^*_{\ell}$ is an independent set, $|B^*_{\ell}|\le n-\delta(G)\le a$.
Therefore,
    \begin{align*}
      |N(v)\cap B^*_j|
      \ge\Big(\frac{2r-5}{2r-3}-\beta\Big)n-5r^2\sqrt{\beta}n-5r\sqrt{\beta}n-\Big(\frac{2r-6}{2r-3}+r\beta\Big)n
      \ge\Big(\frac{1}{2r-3}-12r^2\sqrt{\beta}\Big)n,
  \end{align*}
  as claimed.
\end{pf}

We need to strengthen~\cref{lem:A-structure} as follows.

\begin{claim}\label{cl:AA-structure}
For every connected component $A'\subseteq A^*$, there exists a unique index $i_{A'}\in [r-2]$ such that for every edge $xy\in E(G[A'])$, $N(\{x,y\})\cap B_{i_{A'}}^*=\varnothing$.
\end{claim}
\begin{pf}
    Fix a component $A'$ in $A^*$, we are done by~\cref{lem:A-structure} if $A'$ is a single edge. We may then assume that there is a 3-vertex path $xyz$ in $A'$. Let $i_{xy}$ and $i_{yz}$ be the indices obtained from~\cref{lem:A-structure} such that 
    $N(\{x,y\})\cap B_{i_{xy}}^*=\varnothing$ and $N(\{y,z\})\cap B_{i_{yz}}^*=\varnothing$. It suffices to prove that $i_{xy}=i_{yz}=i_{A'}$ as then we can take a spanning tree of $A'$ and propagate the index $i_{A'}$.
    
    Suppose for a contradiction that $i_{xy}\neq i_{yz}$. As $N(\{x,y\})\cap B_{i_{xy}}^*=\varnothing$, we see that $y$ has at least $|N(x)\cap B^*_{i_{xy}}|
      \ge \big(\frac{1}{2r-3}-12r^2\sqrt{\beta}\big)n$ non-neighbors in $B^*_{i_{xy}}$ by \cref{cl:A-deg}. Similarly, $y$ has at least $|N(z)\cap B^*_{i_{yz}}|
      \ge \big(\frac{1}{2r-3}-12r^2\sqrt{\beta}\big)n$ non-neighbors in $B^*_{i_{yz}}$. On the other hand, $y\in A^*=R_{r-1}$, and so $|N(y)\cap B^*_{r-1}|\le 5r\sqrt{\beta}n$. Altogether, $y$ has at least $\big(\frac{3}{2r-3}-50r^2\sqrt{\beta}\big)n$ non-neighbors, contradicting the minimum degree condition of $G$. 
\end{pf}
   
Fix the unique index $i_{A'}\in[r-2]$ for $A'$ guaranteed by~\cref{cl:AA-structure}. For every $v\in A'$, we can define $\phi(v)$ to be an arbitrary $b$-subset in $N(v)\cap B^*_{i_{A'}}$.
Therefore, it follows from \cref{cl:AA-structure}  that for every edge $uv\in E(G[A'])$, $\phi(u)\cap\phi(v)=\varnothing$ as desired, hence $\phi$ is a $b$-fold coloring on $B_{i_{A'}}^*$. 

Notice that $|B_j^*|\le a$ for any $j\in [r-2]$, so for every connected component $A'\subseteq A^*$, we have $\chi_f(G[A'])\le\frac{a}{b}$. Consequently, for sufficiently large $n$, $\chi_f(G[A^*])\le\frac{a}{b}\le 2+50r^3\sqrt{\beta}$, which proves the first half of part (iii) and implies that $G[A^*]\xrightarrow{\textup{hom}} \mathrm{KN}(a,b)$. The second half follows from the lemma below. Indeed, applying~\cref{lem:hom} with $\delta=\sqrt{\beta}$ and $m=\ln n/\beta$, we have $\frac{b}{a}-\delta\ge \frac{1}{2}-\varepsilon$ and so that $G[A^*]\xrightarrow{\textup{hom}} \mathrm{KN}(m,(1/2-\varepsilon)m)$ as desired.

This completes the proof of \cref{thm:Kr-stability-precise}.
\end{proof}

\begin{lemma}\label{lem:hom}
    Let $G$ be a graph on $n$ vertices. If $G\xrightarrow{\textup{hom}} \mathrm{KN}(a,b)$, then for every $\delta\in(0,b/a)$ and $m>\ln n/(2\delta^2)$ we have $G\xrightarrow{\textup{hom}} \mathrm{KN}(m,(b/a-\delta)m)$.
\end{lemma}

To prove \cref{lem:hom}, we recall the hypergeometric distributions. Suppose that there is a finite set containing $n$ elements, exactly $k\le n$ of which are marked, and suppose that we sample $s\le n$ elements uniformly at random without replacement. Let $X$ be the random variable that counts the number of marked elements in the sample. Then we say that $X$ has the \emph{hypergeometric distribution} with parameters $(n,k,s)$. We will need the following standard tail bound on hypergeometric distributions.

\begin{lemma}[\cite{hoeffding1963probability}]\label{lem:prob}
Let $X$ have the hypergeometric distribution with parameters $(n,k,s)$. Then for $\delta\in(0,k/n)$ the following inequality holds
\begin{align*}
\Pr\left[X\le\Big(\frac{k}{n}-\delta\Big)s\right]\le e^{-2\delta^2 s}.
\end{align*}
\end{lemma}

\begin{proof}[Proof of Lemma \ref{lem:hom}]
Let $\phi:V(G)\rightarrow \binom{[a]}{b}$ be a homomorphism from $G$ to $\mathrm{KN}(a,b)$. Let $S$ be a uniformly chosen random $m$-subset of $[a]$. For each vertex $v\in V(G)$, define a random variable $X_v=|\phi(v)\cap S|$. Clearly, $X_v$ follows the hypergeometric distribution with parameters $(a,b,m)$. Then, by \cref{lem:prob}, we have that for every $v\in V(G)$,
\begin{align*}
    \Pr[X_v\le (b/a-\delta)m]\le e^{-2\delta^2 m}<1/n. 
\end{align*}
By the union bound, there exists a choice of $S\in\binom{[a]}{m}$ such that for all $v\in V(G)$,
\begin{align*}
    (b/a-\delta)m<X_v=|\phi(v)\cap S|\le |S|=m.
\end{align*}
For each $v\in V(G)$, define $\psi(v)$ to be an arbitrary $((b/a-\delta)m)$-subset of $\phi(v)\cap S$. By the previous discussion, $\psi(v)$ is well-defined. Moreover, for every $uv\in E(G)$, we have 
\begin{align*}
    \psi(u)\cap\psi(v)\subseteq(\phi(v)\cap S)\cap(\phi(v)\cap S)\subseteq\phi(v)\cap\phi(u)=\varnothing,
\end{align*}
where the last equality holds by the property of $\phi$. Thus, $\psi:V(G)\rightarrow\binom{S}{(b/a-\delta)m}$ is indeed a homomorphism from $G$ to $\mathrm{KN}(m,(b/a-\delta)m)$, as needed.
\end{proof}

\section{Stability for \texorpdfstring{$\delta_{\chi}(H)=\frac{r-3}{r-2}$}{(r-3)/(r-2)}}\label{sesc:theta-stability}

In this section, we prove the following theorem. 
\begin{theorem}\label{thm:theta-stability-precise-version} 
Let $r\ge 4$, $0<\beta<\frac{1}{30r^3}$ and $H$ be a graph with $\chi(H)=r$ and $\delta_\chi(H)=\frac{r-3}{r-2}$. Then there exists $C=C(H,\beta)$ such that the following holds.
For every $n$-vertex $H$-free graph $G$ with $\delta(G)\ge\big(\frac{r-3}{r-2}-\beta\big)n$ and $\chi(G)\ge C$ there exists
an induced subgraph $G'\subseteq G$ on $n'=\frac{r-3}{r-2}n$ vertices such that
\begin{equation}\label{eq:Gprime}
    |E(G') \Delta E(T_{n', r-3})| \leq 30r^3\beta n'^2.
\end{equation}
\end{theorem}

Note that~\cref{thm:theta-stability-precise-version} implies~\cref{thm:theta-stability}. Indeed, let $S=V(G)\setminus V(G')$, then $|S|=\frac{n}{r-2}$, and we can derive from~\eqref{eq:Gprime} and the min-degree condition $\delta(G)\ge\big(\frac{r-3}{r-2}-\beta\big)n$ that every vertex in $G'$ is almost complete to $S$. Hence, letting $\hat{G}$ be the spanning subgraph of $G$ obtained by deleting all edges induced by $S$, we see that $\hat{G}$ is in edit-distance $o(n^2)$-close to $T_{n,r-2}$ as desired.

We begin by stating an embedding lemma, which is a consequence of Lemmas 24, 25, and Propositions 26, 36 in \cite{allen2013chromatic}. A formal proof is included in the appendix for completeness.

\begin{lemma}[\cite{allen2013chromatic}]\label{lmm:lmm5.3}
Let $H$ be an $r$-near-acyclic graph, and let $\varepsilon, \beta, d > 0$ be real numbers. Suppose that $G$ is a graph and that $X, Y, Z_1, \ldots, Z_{r-3}$ are pairwise disjoint subsets of $V(G)$ such that $|Y| = |Z_j|$ for each $j \in [r-3]$. Assume that the pairs $(Y, Z_j)$ and $(Z_i, Z_j)$ are $(\varepsilon, d)$-regular for all distinct $i, j$, and
\begin{align*}
|N(x) \cap Y| \geq \beta |Y| \quad \text{and} \quad |N(x) \cap Z_j| \geq (1/2 + \beta) |Z_j|
\end{align*}
for every $x \in X$ and every $j \in [r-3]$. If $\chi(G[X]) \geq C$ for some sufficiently large constant $C = C(H, \varepsilon, \beta, d)$, then $H \subseteq G$.
\end{lemma}

\begin{proof}[Proof of \cref{thm:theta-stability-precise-version}]
Fix $r\ge 4$, $0<\beta<\frac{1}{30r^3}$ and an $r$-near-acyclic graph $H$.  Let $C$ be a sufficiently large integer whose value will be determined later.
Let $G$ be an $H$-free graph on $n$ vertices satisfying $\delta(G)>(\frac{r-3}{r-2}-\beta)n$ and $\chi(G)\ge C$. Set $d := \frac{\beta}{20r}$ and $\varepsilon := \frac{d^2}{2d + 2|H|}$.
Applying \cref{thm:RL} to $G$ with $k_0=\frac{r}{8\beta}$, we obtain a partition $V(G) = V_0 \cup V_1 \cup \cdots \cup V_k$ , together with an $(\varepsilon,d)$-reduced graph $R$ with vertex set $[k]$, where $\frac{r}{8\beta} \leq k\leq k_1$ for some $k_1=k_1(\varepsilon,d,k_0)$. Furthermore, the minimum degree of $R$ satisfies $\delta(R) \ge \big(\frac{r-3}{r-2} - 2\beta\big)k$.

We now define a refined partition of $V(G)$. For each pair of index sets $I_2 \subseteq I_1 \subseteq [k]$, let
\begin{align*}
X(I_1, I_2):=  \{v \in V(G): i \in I_1 \Leftrightarrow|N(v) \cap V_i| \geq \beta|V_i| \text { and } 
 i \in I_2 \Leftrightarrow|N(v) \cap V_i| \geq(1/2+\beta)|V_i|\}.
\end{align*} 
That is, $X(I_1, I_2)$ consists of all vertices $v \in V(G)$ such that for each $i \in [k]$, $i \in I_1$ if and only if $v$ has at least $\beta|V_i|$ neighbors in $V_i$, and
$i \in I_2$ if and only if $v$ has at least $(1/2+\beta)|V_i|$ neighbors in $V_i$.

We first establish a lower bound for $|I_1| + |I_2|$ when $X(I_1, I_2) \ne \varnothing$.
\begin{claim}\label{I1+I2}
If $X(I_1,I_2)\neq\varnothing$, then $|I_1|+|I_2|\ge 2\big(\frac{r-3}{r-2}-4\beta\big)k.$
\end{claim}
\begin{pf}
 For every vertex $x\in X(I_1, I_2)$, since $x$ has at most $(\varepsilon+\beta)n$ neighbors outside $V_{I_1}$,
\[|N(x)\cap V_{I_1}| \geq d(x)-(\varepsilon+\beta)n\geq \Big(\frac{r-3}{r-2}- 3\beta\Big) n.\]
On the other hand,
\[|N(x)\cap V_{I_1}|\le \frac{n}{k}\big((1/2+\beta)(|I_1|-|I_2|)+|I_2|\big)\le \beta n+\frac{n}{2k}(|I_1|+|I_2|).\]
Comparing the upper and lower bounds on $|N(x) \cap V_{I_1}|$ yields the desired inequality.
\end{pf}

Next, for those index sets $I_2 \subseteq I_1 \subseteq [k]$ with $X(I_1, I_2) \ne \varnothing$, we will show that the chromatic number $\chi(G[X(I_1, I_2)])$ is bounded whenever either $R[I_1]$ contains a copy of $K_{r-1}$ or $R[I_2]$ contains a copy of $K_{r-2}$.

\begin{claim}\label{claim:property-I_1-and-I_2}
There exist $C_1=C_1(H,\beta)$ and $C_2=C_2(H,\beta)$ such that for $X(I_1,I_2)\neq\varnothing$, the following hold.
\begin{itemize}
\item [{\rm (i)}] If $R[I_1]$ contains a copy of $K_{r-1}$, then $\chi(G[X(I_1,I_2)])\le C_1$;
\item [{\rm (ii)}] If $R[I_2]$ contains a copy of $K_{r-2}$, then $\chi(G[X(I_1,I_2)])\le C_2$.
\end{itemize}
\end{claim}

\begin{pf}
Since $H$ is $r$-near-acyclic, there exists a forest $F$ such that $H\subseteq F\vee K_{r-2}[|H|]$. 

For part (i), suppose that $R[I_1]$ contains a copy of $K_{r-1}$, without loss of generality, assume its vertex set is $[r-1]\subseteq I$.
 For any $x \in X(I_1, I_2)$, we have $|N(x) \cap V_i| \ge d |V_i|$ for all $i \in [r-1]$.
By \cref{fact:slicinglemma}, it follows that for every pair $\{i, j\} \in \binom{[r-1]}{2}$, the pair $(N(x) \cap V_i, N(x) \cap V_j)$ is $(\varepsilon/d, d - \varepsilon)$-regular. Hence, by \cref{thm:counting}, $G[N(x)]$ contains at least $\alpha_1 n^{(r-1)|H|}$ copies of $K_{r-1}[|H|]$ for some $\alpha_1=\alpha_1(H,\beta)>0$. 
  Since $G$ is $H$-free, it follows from \cref{lmm:lmm9}(b) that $\chi(G[X(I_1,I_2)])\leq |X(I_1,I_2)|\leq |H| / \alpha_1.$
  By letting $C_1(H,\beta)=|H| / \alpha_1$, this proves part (i).
  
\vskip 0.5em

 For part (ii), suppose that $R[I_2]$ contains a copy of $K_{r-2}$ with vertex set $[r-2] \subseteq I_2$. For any edge $xy\in E(G[X(I_1,I_2)])$, define $V_i'=N(\{x,y\})\cap V_i$ for each $i\in [r-2]$. Then \[|V_i'|\geq |N(x)\cap V_i|+|N(y)\cap V_i|-|V_i|\geq 2\beta|V_i|.\] By \cref{fact:slicinglemma},  each pair $(V_i',V_j')$ is $(\varepsilon /(2\beta),d-\varepsilon)$-regular for every $\{i,j\}\in\binom{[r-2]}{2}$. 
 Again by \cref{thm:counting}, $G[\bigcup_{i=1}^{r-2} V_i']$ contains at least $\alpha' n^{r-2}$ copies of $K_{r-2}$ for some $\alpha'=\alpha'(H,\beta)>0$. Therefore, by \cref{lmm:lmm10}, there exists a constant $\alpha_2=\alpha_2(\alpha', H)$ such that $\chi(G[X(I_1,I_2)])\leq 2|H|/\alpha_2+1$.    
 By letting $C_2(H,\beta)=2|H|/\alpha_2+1$, part (ii) is thus proved.
\end{pf}

Let $C_3$ be $C(H,\varepsilon,\beta,d)$ given from \cref{lmm:lmm5.3} and let $C=C(H,\beta)= 5^k \max\{C_1, C_2, 2C_3\}$.
Then there exists a pair  $I_2\subseteq I_1\subseteq [k]$ such that $\chi(G[X(I_1,I_2)])\ge C/4^k >\max\{C_1,C_2\}$.
Fix such a pair $(I_1, I_2)$. 
By \cref{claim:property-I_1-and-I_2}, we must have that $R[I_1]$ is $K_{r-1}$-free and $R[I_2]$ is $K_{r-2}$-free.

We now split the remainder of the proof into two cases, according to whether $R[I_2]$ contains a copy of $K_{r-3}$. 
Recall that $\delta(R)\geq \big(\frac{r-3}{r-2}-2\beta\big)k$,
hence any vertex in $R$ has at most $(\frac{1}{r-2}+2\beta)k$ non-neighbors. Consequently, any independent set in $R$ has size at most $\big(\frac{1}{r - 2} + 2\beta \big)k$.

\paragraph{Case 1: $R[I_2]$ is $K_{r-3}$-free.}In this case, we establish a stronger structural result that 
$$|E(G)\triangle E(T_{n,r-2})|\le  30r^3\beta n'^2.$$ 
We first claim $\chi(R[I_2])\leq r-4$. If $r\in \{4,5\}$, it is clear. Assume $r\geq 6$.
By \cref{I1+I2}, $|I_1|+|I_2|\ge 2(\frac{r-3}{r-2}-4\beta)k$, hence $|I_2|\ge (\frac{r-4}{r-2}-8\beta)k$. Therefore,
$$
\delta(R[I_2])\ge |I_2|-(k-\delta(R)) \geq|I_2|-\Big(\frac{1}{r-2}+2\beta\Big) k>\frac{3r-16}{3r-13}|I_2|.
$$
Hence, by \cref{thm:1974Andrasfai}, $R[I_2]$ is $(r-4)$-colorable.
Thus, for all $r\ge 4$, $I_2$ can be partitioned into $r-4$ disjoint independent sets. 
Hence $|I_2| \le (r-4)(\frac{1}{r-2}+2\beta) k.$ 

We can then again apply~\cref{I1+I2} to get $|I_1|\geq (1-2r\beta)k,$ which further implies that
\begin{align*}
  \delta(R[I_1]) \geq|I_1|-\Big(\frac{1}{r-2}+2\beta\Big) k>\frac{3r-10}{3r-7}|I_1|.
\end{align*}
Since $R[I_1]$ is $K_{r-1}$-free, it follows from \cref{thm:1974Andrasfai} that $R[I_1]$ is $(r-2)$-colorable. 
Hence, $I_1$ can be partitioned as $I_1 = W_1 \cup \cdots \cup W_{r - 2}$, where each $W_i$ is an independent set in $R$. We now estimate $|W_i|$ by observing that
\begin{align*}
\Big(\frac{1}{r-2}+2\beta\Big)k\ge |W_i|\ge |I_1|-(r-3)\Big(\frac{1}{r-2}+2\beta\Big) k\ge \Big(\frac{1}{r-2}-4r\beta\Big)k.
\end{align*}
Thus, $||W_i|-\frac{k}{r-2}|\leq 4r\beta k$ for each $i\in [r-2]$.

Moreover, as each $W_i$ is independent, each vertex in $W_i$ has at most $\big(\frac{1}{r - 2} + 2\beta\big)k - |W_i| \le (4r + 2)\beta k$ non-neighbors outside $W_i$.
Thus, the total number of missing edges across all distinct pairs $(W_i, W_j)$ is at most $(2r + 1)\beta k^2$. Consequently, we obtain
\begin{align*}
|E(R)\triangle E(T_{k,r-2})|\le k(k-|I_1|)+(2r+1)\beta k^2+k\cdot\Big(\sum_{i=1}^{r-2}\Big||W_i|-\frac{k}{r-2}\Big|\Big)\le(4r^2+1)\beta k^2.
\end{align*}
It follows by \cref{lmm:edit} that $|E(G)\triangle E(T_{n,r-2})|\le (8r^2+3)\beta n^2\le 4(8r^2+3)\beta n'^2\le 30r^3\beta n'^2$.

\paragraph{Case 2: $R[I_2]$ contains at least one copy of $K_{r-3}$.}
Suppose that $\{z_1,\ldots,z_{r-3}\}\subseteq I_2$ forms a copy of  $K_{r-3}$ in $R[I_2]$. We further divide the remaining of the proof into two subcases, depending on whether $N(\{z_1,\ldots,z_{r-3}\})\cap I_1=\varnothing$.

\paragraph{Subcase 2.1: $N(\{z_1,\ldots,z_{r-3}\})\cap I_1=\varnothing$.} We claim that $R[I_2]$ is $(r-3)$-colorable.  As $R[I_2]$ is $K_{r-2}$-free, this is clear when $r=4$. We may assume that $r\geq 5$. Note that 
\begin{align*}
    |N(\{z_1,\ldots,z_{r-3}\})|\ge k-(r-3)(k-\delta(R))\ge\Big(\frac{1}{r-2}-2\beta(r-3)\Big)k.
\end{align*}
Hence $|I_1|\leq(\frac{r-3}{r-2}+2\beta(r-3))k$. By \cref{I1+I2}, we have
\[\Big(\frac{r-3}{r-2}-3r\beta\Big)k<|I_2|\le |I_1|\le \Big(\frac{r-3}{r-2}+2\beta(r-3)\Big)k.\]
As each vertex in $R$ has at most $(\frac{1}{r-2}+2\beta) k$ non-neighbors,
\begin{align*}
\delta(R[I_2]) \geq|I_2|-\Big(\frac{1}{r-2}+2\beta\Big) k >\frac{3r-13}{3r-10}|I_2|.
\end{align*}
Then it follows from \cref{thm:1974Andrasfai} that $R[I_2]$ is $(r-3)$-colorable as claimed. Therefore, $I_2$ admits a partition $I_2=W_1\cup\cdots\cup W_{r-3}$ where each  $W_i$ is an independent set in $R$ for any $i\in [r-3]$. Hence, 
\begin{align*}
\left(\frac{1}{r-2}+2\beta\right) k\geq |W_i|\geq |I_2|-(r-4)\Big(\frac{1}{r-2}+2\beta\Big) k\geq \Big(\frac{1}{r-2}-(5r-8)\beta\Big)k.
\end{align*}
Then, for each $i\in [r-3]$, we have $||W_i|-\frac{k}{r-2}|\leq 5r\beta k$.

Moreover, since each vertex in  $W_i$ has at most $(\frac{1}{r-2}+2\beta) k-|W_i|\le (5r-6)\beta k$ non-neighbors outside $W_i$, the total number of missing edges among all distinct pairs $W_i$ and $W_j$ is at most $(5r-6)\beta k^2/2$. Let $R'=R[I_2]$ and $k'=\frac{r-3}{r-2}k$, $G'=G[\bigcup_{i\in I_2} V_i]$ and $n'=\frac{r-3}{r-2}n$.  We conclude that 
\begin{align*}
|E(R')\triangle E(T_{k',r-3})|\le k|k'-|I_2||+\frac{(5r-6)\beta k^2}{2}+k\cdot\Big(\sum_{i=1}^{r-3}\Big||W_i|-\frac{k}{r-2}\Big|\Big)\le 5r^2\beta k^2.
\end{align*}
It follows by \cref{lmm:edit} that $|E(G')\triangle E(T_{n',r-3})|\le  30r^3\beta  n'^2$.

\paragraph{Subcase 2.2: $N(\{z_1,\ldots,z_{r-3}\})\cap I_1\neq\varnothing$.} Pick an arbitrary vertex $y\in N(\{z_1,\ldots,z_{r-3}\})\cap I_1$. 
 Let $I_1'=I_1\setminus \{y,z_1,\ldots,z_{r-3}\}$ and $I_2'=I_2\setminus \{y,z_1,\ldots,z_{r-3}\}$.

If $R[I_2']$ is $K_{r-3}$-free, the proof is almost identical to that in Case 1 (with only the change of replacing $(I_1,I_2)$ by $(I_1',I_2')$). Now, suppose that $R[I_2']$ contains a copy of $K_{r-3}$ on the vertices $\{z_1',\ldots,z_{r-3}'\}\subseteq I_2'$. Since $G$ is $H$-free, the following claim is an immediate consequence of \cref{lmm:lmm5.3}.

\begin{claim}\label{cl:2K}
$N(\{z_1',\ldots,z_{r-3}'\})\cap I_1'=\varnothing$.
\end{claim}

\begin{pf}
Suppose for contradiction that there exists $y' \in I_1'$ such that  $\{y',z_1',\ldots,z_{r-3}'\}$ forms a copy of $K_{r-2}$ in $R$. Let $Y,Y',Z_1,Z_1',\ldots,Z_{r-3},Z_{r-3}'$ be the clusters of $V(G)$ corresponding to  the vertices $y, y',z_1,z_1',\ldots,z_{r-3},z_{r-3}'$ in $R$ respectively, and let
\begin{align*}
  X=X(I_1, I_2) \cap(Y' \cup Z_1' \cup \cdots \cup Z_{r-3}') \text{\quad and \quad} X'=X(I_1, I_2) \setminus X. 
\end{align*}

It is routine to verify by definition that the two  sets $\{X, Y, Z_1, \ldots, Z_{r-3}\}$ and $\{X', Y', Z_1', \ldots, Z_{r-3}'\}$ both satisfy the assumptions of \cref{lmm:lmm5.3}. 
Since $G$ is $H$-free, it follows from \cref{lmm:lmm5.3} that $\max\{\chi(G[X]), \chi(G[X'])\} \le C_3$, which contradicts the assumption that $\chi(X(I_1,I_2))\geq C/4^k >2C_3$.
\end{pf}

Then, using essentially the same argument as in Subcase 2.1, with $I_2'$ in place of $I_2$, one can show that the conclusion of Subcase 2.1 also holds in this case.
\end{proof}

\section{Fractional chromatic threshold}\label{fractional}

In this section, we aim to prove \cref{thm:fractional-chromatic-threshold-2}. First, consider the case when $\mathcal{M}(H)$ contains no forest. As in \cref{construction:extremal-graph-pi}, let  $G=G'\vee K_{r-2}[|G'|]$, where $G'$ is a $(C,|H|+1)$-Erd\H{o}s graph. Then, $G$ is an $H$-free graph with $\delta(G)=\frac{r-2}{r-1}\cdot |G|$ and $\chi_f(G) \geq C$, which implies that $\delta_{\chi_f}(H)\ge\frac{r-2}{r-1}$. On the other hand, $\delta_{\chi_f}(H)$ is trivially upper bounded by the Tur\'an density of $H$ and so by Erd\H{o}s–Stone Theorem \cite{erdos1966limit}, we have $\delta_{\chi_f}(H)\le \frac{r-2}{r-1}$. Thus, we have the following proposition.

\begin{proposition}
  Let $H$ be a graph with $\chi(H)=r \geq 3$. If $\mathcal{M}(H)$ does not contain a forest, then $\delta_{\chi_f}(H)=\frac{r-2}{r-1}$.
\end{proposition}

Next, we turn to the case where $\mathcal{M}(H)$ contains a forest. Using \cref{construction:extremal-graph-theta}, i.e., let $G=G'\vee K_{r-3}[|G'|]$, where $G'$ is a $(C,|H|+1)$-Erd\H{o}s graph, we obtain an $H$-free graph $G$ with $\delta(G)=\frac{r-3}{r-2}\cdot |G|$ and $\chi_f(G) \geq C$, which gives the lower bound $\delta_{\chi_f}(H)\ge\frac{r-3}{r-2}$.

It remains to show the following theorem to obtain the corresponding upper bound.

\begin{theorem}\label{thm:fractional-r}
  Let $H$ be a graph with $\chi(H)=r\ge 3$ and suppose $\mathcal{M}(H)$ contains a forest. Then for any $n$-vertex $H$-free graph $G$ with $\delta(G)\ge (\frac{r-3}{r-2}+\varepsilon)n$, we have $\chi_f(G)=O_{\varepsilon,H}(1)$. 
\end{theorem}

\begin{proof}
Suppose that $H\subseteq F\vee K_{r-2}[t]$ for some forest $F\in \mathcal{M}(H)$. Let $G$ be an $H$-free graph with $\delta(G)\ge (\frac{r-3}{r-2}+\varepsilon)n$. We shall bound its fractional chromatic number.

    If $r\ge 4$, then for any vertex $v\in V(G)$ we have
\begin{align*}
    \frac{\delta(G[N(v)])}{d(v)}\ge \frac{d(v)-(n-d(v))}{d(v)}\ge \frac{(\frac{r-3}{r-2}+\varepsilon)n-(\frac{1}{r-2}-\varepsilon)n}{(\frac{r-3}{r-2}+\varepsilon)n}>\frac{r-4}{r-3}+\varepsilon.
\end{align*}
By Erd\H{o}s-Simonovits supersaturation theorem \cite{erdHos1983supersaturated}, there exists at least $\alpha n^{(r-2)t}$ copies of $K_{r-2}[t]$ in $G[N(v)]$ for some constant $\alpha=\alpha(\varepsilon)$. Note that the statement is also true for $r=3$.

For each copy $T$ of $K_{r-2}[t]$ in $G$, $G[N(T)]$ must be $F$-free since $G$ is $H$-free. 
By \cref{fact:forest-free}, $G[N(T)]$ is $|F|$-colorable, as $F$ is a forest. Thus, $N(T)$ can be partitioned into $s$ disjoint independent sets, where $s=\chi(G[N(T)])\le |F|\le|H|$. Label these independent sets as $T(1),T(2),\ldots,T(s)$. 
Then, for any vertex $v$ and a copy $T'$ of  $K_{r-2}[t]$ in $G[N(v)]$, we assign  to $v$  the color $(T',i)$ where the independent set $T'(i)$ containing $v$. Since $G$ contains at most $n^{(r-2)t}$ copies of $K_{r-2}[t]$ and each copy of $K_{r-2}[t]$ defines at most $|H|$ independent sets in its neighborhood, every vertex received at least $\alpha n^{(r-2)t}$ colors from a palette of size at most $a:=|H|n^{(r-2)t}$.

Let $b=\alpha n^{(r-2)t}$. Now we construct an $a:b$-coloring of $V(G)$. For each vertex $v$, arbitrarily select $b$ colors among those assigned to it. We now show that this is a valid $b$-fold coloring of $G$. Indeed, a color is assigned to both vertices $u$ and $v$ if and only if there exist a copy $T$ of $K_{r-2}[t]$ and $1\le i\le |H|$ such that $V(T)\subseteq N(u)\cap N(v)$ and $u, v \in T(i)$, which implies that $u$ and $v$ are non-adjacent. Thus $$\chi_f(G)\le \frac{a}{b}\le \frac{|H|n^{(r-2)t}}{\alpha n^{(r-2)t}}=O_{\varepsilon,H}(1),$$
as desired.
\end{proof}

\section{Bounded-VC chromatic threshold}\label{sec:bounded-VC}

In this section, we prove \cref{thm:bounded-VC}. Given a set system $\mathcal{F}\subseteq 2^{V}$, the \emph{Vapnik-Chervonenkis dimension} (\emph{VC-dimension} for short) of $\mathcal{F}$ , denoted by $\mathrm{VC}(\mathcal{F})$, is the largest integer $d$ such that there exists a subset $X\subseteq V$ with $|X|=d$ satisfying the following condition: for every subset $Y\subseteq X$, there exists $F\in\mathcal{F}$ with $F\cap X=Y$. In this case, we say that $X$ is \emph{shattered} by $\mathcal{F}$. 
We define the VC-dimension of a graph $G$, denoted by $\mathrm{VC}(G)$, to be the VC-dimension of the set system induced by the neighborhoods of its vertices, that is, $\mathrm{VC}(G):=\mathrm{VC}(\mathcal{F})$, where $\mathcal{F}=\{N(v):v\in V(G)\}$.

\subsection{Proof of the lower bound}

To prove the lower bound $\delta_\chi^{\text{VC}}(K_r)\ge\frac{r-3}{r-2}$, we aim to show that for any integer $C$, there exists an infinite sequence of $n$-vertex $K_r$-free graphs $G_n$ with $\mathrm{VC}(G_n)= 2$, $\delta(G_n)\ge \frac{r-3}{r-2} n$, and $\chi(G_n)\ge C$. For this purpose, we will employ a classical construction due to  Erd\H{o}s and Hajnal \cite{MR0263693}.

\begin{definition}[Shift Graph]\label{def:shift-graphs}
For positive integers $m>k\ge 2$, the shift graph $\mathrm{Sh}_m^k$ defined on the vertex set $V(\mathrm{Sh}_m^k)=\{(x_1, \ldots, x_k): 1 \leq x_1<\cdots<x_k \leq m\}$ is a graph in which two vertices $\bm{x}=(x_1, \ldots, x_k)$ and $\bm{y}=(y_1, \ldots, y_k)$ are adjacent if and only if $x_i=y_{i+1}$ for all $i \in [k-1]$.  
\end{definition}

\begin{figure}[H]
    \centering
\begin{tikzpicture}[scale=1.3]
    \coordinate (14) at (1,4);
    \coordinate (24) at (2,4);
    \coordinate (34) at (3,4);
    \coordinate (44) at (4,4);
    \coordinate (23) at (2,3);
    \coordinate (33) at (3,3);
    \coordinate (43) at (4,3);
    \coordinate (32) at (3,2);
    \coordinate (42) at (4,2);
    \coordinate (41) at (4,1);

    \draw [red,line width=1.5pt] (14) -- (23);
    \draw [red,line width=1.5pt] (23) -- (32);
    \draw [red,line width=1.5pt] (32) -- (41);
    \draw [blue,line width=1.5pt] (14) -- (33);
    \draw [blue,line width=1.5pt] (24) -- (32);
    \draw [blue,line width=1.5pt] (23) -- (42);
    \draw [blue,line width=1.5pt] (33) -- (41);
    \draw [green,line width=1.5pt] (34) -- (41);
    \draw [green,line width=1.5pt] (14) -- (43);
    \draw [orange,line width=1.5pt] (24) to[bend left=18] (42);

    \filldraw[fill=white] (14) circle (0.2cm);
    \filldraw[fill=white] (24) circle (0.2cm);
    \filldraw[fill=white] (34) circle (0.2cm);
    \filldraw[fill=white] (44) circle (0.2cm);
    \filldraw[fill=white] (23) circle (0.2cm);
    \filldraw[fill=white] (33) circle (0.2cm);
    \filldraw[fill=white] (43) circle (0.2cm);
    \filldraw[fill=white] (32) circle (0.2cm);
    \filldraw[fill=white] (42) circle (0.2cm);
    \filldraw[fill=white] (41) circle (0.2cm);

    \node at (14) [font=\footnotesize] {1,2};
    \node at (24) [font=\footnotesize] {1,3};
    \node at (34) [font=\footnotesize] {1,4};
    \node at (44) [font=\footnotesize] {1,5};
    \node at (23) [font=\footnotesize] {2,3};
    \node at (33) [font=\footnotesize] {2,4};
    \node at (43) [font=\footnotesize] {2,5};
    \node at (32) [font=\footnotesize] {3,4};
    \node at (42) [font=\footnotesize] {3,5};
    \node at (41) [font=\footnotesize] {4,5};
    
\end{tikzpicture}
    \caption{The shift graph $\mathrm{Sh}_5^2$}
    \label{fig:enter-label}
\end{figure}

 In \cite{MR0263693}, Erd\H{o}s and Hajnal showed that the shift graph $\mathrm{Sh}_m^2$ is $K_3$-free and has chromatic number $(1+o(1))\log_2 m$. 
We now demonstrate that the VC-dimension of $\mathrm{Sh}_m^2$ equals 2.

\begin{proposition}\label{prop:VC}
$\mathrm{VC}(\mathrm{Sh}_m^2)=2$. 
\end{proposition}

\begin{proof}
As depicted in~\cref{fig:enter-label}, the 2-set $\{(2,4),(3,4)\}$ is shattered which establishes the lower bound. To prove the upper bound $\mathrm{VC}(\mathrm{Sh}_m^2)\le 2$, suppose, for contradiction, that there exists a subset $\{(i_1,i_2),(j_1,j_2),(k_1,k_2)\}\subseteq V(\mathrm{Sh}_m^2)$ which is shattered by the neighborhoods of vertices in $\mathrm{Sh}_m^2$. Then, there exists $(x_1,x_2)\in V(\mathrm{Sh}_m^2)$ such that all of $(i_1,i_2),(j_1,j_2),(k_1,k_2)$ are neighbors of $(x_1,x_2)$. By definition, every neighbor of $(x_1,x_2)$ must be of the form either $(a,x_1)$ for some $1\le a<x_1$ or $(x_2,b)$ for some $x_2<b\le m$. By the pigeonhole principle, at least two of the vertices $(i_1,i_2),(j_1,j_2),(k_1,k_2)$ must share a common coordinate structure, say without loss of generality, $(i_1,i_2)=(a,x_1)$ and $(j_1,j_2)=(a',x_1)$ with distinct $a,a'<x_1$. 
   To derive the desired contradiction, it suffices to show that there is no vertex that is adjacent to both $(i_1,i_2)$ and $(k_1,k_2)$, but not to $(j_1,j_2)$.
   Indeed, if $(k_1,k_2)=(x_2,b)$, then $(a,x_1)$ and $(x_2,b)$ has a unique common neighbor $(x_1,x_2)$, which is also adjacent to $(a',x_1)$.
   If $(k_1,k_2)=(a'',x_1)$ for some $a''<x_1$, then the common neighbor of $(a,x_1)$ and $(a'',x_1)$ is of the form $(x_1,c)$ for some $x_1<c$, which is also adjacent to $(a',x_1)$.
\end{proof}

We now present the construction that establishes the lower bound.
\begin{construction}\label{cons:lower-bd:VC}
Given $m>k\ge 2$ and $r\ge 4$, for sufficiently large $n$, let $G'$ be the vertex-disjoint union of $\mathrm{Sh}_m^2$ and $\frac{n}{r-2}-\binom{m}{2}$ isolated vertices. Then let  $G=G'\vee K_{r-3}[|G'|]$.
\end{construction}

Since $G'$ is $K_3$-free, $G$ is $K_r$-free. Moreover, it is not hard to see that any shattered subset of vertices in $G$ intersects at most one vertex part. Thus, we have $\mathrm{VC}(G)=\mathrm{VC}(\mathrm{Sh}_m^2)=2$. To conclude, $G$ is a $K_r$-free graph with $\delta(G)\ge\frac{r-3}{r-2}n$ and $\mathrm{VC}(G)=2$, and $\chi(G)$  grows arbitrarily large as $m$ tends to infinity.

\subsection{Proof of the upper bound}
The upper bound follows from the theorem below.
 \begin{theorem}\label{thm:bdd VC}
 Let $c>0$, and let $G$ be an $n$-vertex $K_r$-free graph with minimum degree $(\frac{r-3}{r-2}+c)n$ and VC-dimension $d$.
 Then $\chi(G) \leq (16d/c\cdot\log(d/c))^{r-2}$.
\end{theorem}

To prove \cref{thm:bdd VC}, we will make use of the celebrated $\varepsilon$-net theorem of Haussler and Welzl \cite{haussler1986epsilon}. We will need some definitions before introducing this result.

A \emph{transversal} of a set system $\mathcal{F}\subseteq 2^X$ is a subset $T\subseteq X$ that intersects every set in $\mathcal{F}$. The \emph{transversal number} of $\mathcal{F}$, denoted by $\tau(\mathcal{F})$, is the smallest size among all transversals of $\mathcal{F}$. A \emph{fractional transversal} of $\mathcal{F}$ is a weight function $\omega: X\rightarrow[0,1]$ defined on the ground set $X$ such that, for every set $A\in\mathcal{F}$, we have $\omega(A):=\sum_{a\in A}\omega(a)\ge 1$. The \emph{fractional transversal number} of $\mathcal{F}$, denoted by $\tau^*(\mathcal{F})$, is defined as the minimum of $\omega(X)=\sum_{x\in X}\omega(x)$ over all fractional transversals $\omega$ of $\mathcal{F}$.

By definitions, $\tau^*(\mathcal{F})\le \tau(\mathcal{F})$. However, the gap between them could be arbitrarily large. Haussler and Welzl \cite{haussler1986epsilon} proved the following lemma, showing that the gap can be bounded if $\mathcal{F}$ has bounded VC-dimension.
\begin{lemma}[\cite{haussler1986epsilon}]\label{lmm:VC-dim}
 Every set system $\mathcal{F}$ with $V C$-dimension $d$ satisfies
$\tau(\mathcal{F}) \leq 16 d \tau^*(\mathcal{F}) \log \left(d \tau^*(\mathcal{F})\right).$
\end{lemma}

\begin{proof}[Proof of \cref{thm:bdd VC}]
By assumption, the set system $\mathcal{F}=\{N(v):v\in G\}$ has VC-dimension $d$. Since every set in $\mathcal{F}$ contains at least $cn$ elements of $V(G)$, it follows that $\tau^*(\mathcal{F})\le 1/c$ as demonstrated by the constant fractional transversal $\omega\equiv \frac{1}{cn}$. Therefore, by \cref{lmm:VC-dim} we can find a transversal $T\subseteq V(G)$ for $\mathcal{F}$ with size  $|T|=\tau(\mathcal{F})\le 16d/c\cdot\log(d/c)$.

We prove \cref{thm:bdd VC} by induction on $r$. Observe that by the definition of transversal,  $V(G)=\bigcup_{v\in T} N(v)$ and so $\chi(G)\le \sum_{v\in T}\chi(G[N(v)])$. For the base case $r=3$, when $G$ is $K_3$-free, $N(v)$ is an independent set for each $v\in T$. Thus $\chi(G)\le |T|\le 16d/c\cdot\log(d/c).$ Now, assume that \cref{thm:bdd VC} holds for every integer $3\leq r'<r$.   
\begin{claim}\label{bdd-chi}
 For every vertex $v\in V(G)$, $\chi(G[N(v)])\le (16d/c\cdot\log(d/c))^{r-3}$. 
\end{claim}

\begin{pf}
Let $v\in V(G)$ and define $G':=G[N(v)]$, then $|G'|=d(v)\ge (\frac{r-3}{r-2}+c)n$.
Since $G$ is $K_{r}$-free, it follows that $G'$ is $K_{r-1}$-free. Thus, for every vertex $w\in V(G')$, we have 
\begin{align*}
 d_{G'}(w)\ge d_G(w)-(n-|G'|)\ge  \Big(\frac{r-3}{r-2}+c\Big)n-n+|G'|\geq \Big(\frac{r-4}{r-3}+c'\Big)|G'|,
\end{align*}
where $c'=\frac{c(r-2)^2}{(r-3)(c(r-2)+r-3)}>c>0$. Since $\mathrm{VC}(G')\le\mathrm{VC}(G)\le d$, by the induction hypothesis, we obtain $\chi(G')\leq (16d/c\cdot\log(d/c))^{r-3}$.
\end{pf}

Recall that $\chi(G)\le \sum_{v\in T}\chi(G[N(v)])$.
Thus, by~\cref{bdd-chi}, we have
\[\chi(G)\le |T|\cdot (16d/c\cdot\log(d/c))^{r-3} \le (16d/c\cdot\log(d/c))^{r-2},\]
as desired.
\end{proof}

\section*{Acknowledgments}
We would like to thank Rob Morris for fruitful and enlightening discussions at the early stage of this project. We are also grateful to the referee for a careful reading of the paper and for helpful suggestions that improved both the clarity and the exposition.

\section*{Declaration of No Conflict}
The authors declare that they have no conflicts of interest.

\bibliographystyle{abbrv}
\bibliography{references.bib}

\appendix
\section{Appendix: Proof of Lemma \ref{lmm:lmm5.3}}

To embed near-acyclic graphs, we instead embed so-called Zykov graphs which are universal graphs for near-acyclic graphs.
Let $K(v, X)$ denote the set of pairs $\{vx: x \in X\}$.

\begin{definition}[Modified Zykov Graphs \cite{allen2013chromatic}]
 Let $T_1, \ldots, T_{\ell}$ be disjoint trees, where $T_j$ has bipartition $A^{0}_j \cup A^{1}_j$, $j\in[\ell]$. 
 The graph $Z_{\ell}(T_1, \ldots, T_{\ell})$ is constructed as follows: 
$$
V(Z_{\ell}(T_1, \ldots, T_{\ell})):=\left(\bigcup_{j \in[\ell]} A^{0}_j \cup A^{1}_j\right) \cup\left\{u_I: I \subseteq[\ell]\right\}
$$
with edge set
$$
E(Z_{\ell}(T_1, \ldots, T_{\ell})):=\bigcup_{j=1}^{\ell}\left(E\left(T_j\right) \cup \bigcup_{j \in I \subseteq[\ell]} K(u_I, A^{0}_j) \cup \bigcup_{j \notin I \subseteq[\ell]} K(u_I, A^{1}_j)\right) .
$$
For each $r \geq 3$ and $t \in \mathbb{N}$, the modified Zykov graph $Z_{\ell}^{r, t}(T_1, \ldots, T_{\ell})$ is the graph obtained from $Z_{\ell}(T_1, \ldots, T_{\ell})$ by performing the following two operations:
\begin{itemize}
    \item[\rm (a)] Add vertices $W=\{w_1, \ldots, w_{r-3}\}$, and all edges with an endpoint in $W$.
    \item[\rm (b)] Blow up each vertex $u_I$ with $I \subseteq[\ell]$ to a set $S_I$ of size $t$ and each vertex $w_j$ in $W$ to a set $S_j^{\prime}$ of size $t$.
\end{itemize}
\end{definition}
Finally, we shall write $Z_{\ell}^{r, t}$ for the modified Zykov graph obtained when each tree $T_i, i \in[\ell]$, is a single edge; that is, $Z_{\ell}^{r, t}=Z_{\ell}^{r, t}(e_1, \ldots, e_{\ell})$.

\begin{observation}[\cite{allen2013chromatic}]
   Let $\chi(H)=r$. Then $H$ is $r$-near-acyclic if and only if there exist trees $T_1, \ldots, T_{\ell}$ and $t \in \mathbb{N}$ such that $H$ is a subgraph of $Z_{\ell}^{r, t}(T_1, \ldots, T_{\ell})$. 
\end{observation}

Let $\ell$ be an integer.
  For each $I\subseteq [\ell]$, let $\bm{u}(I) = \{u_1,\cdots, u_\ell\}\in \{0,1\}^\ell$ with $u_i=1$ for $i\in I$ and $u_i=0$ for $i\notin I$.
  For convenience, for any $\bm{u}(I) = \{u_1,\cdots, u_\ell\}\in \{0,1\}^\ell$, let $A^{\bm{u}(I)}=\bigcup_{i=1}^\ell A^{u_i}_i.$
  
 We restate Lemma \ref{lmm:lmm5.3} with more details as follows.

\begin{lemma}\label{lmm:0-1/2}
     Let $\ell,t$ be two positive integers, and $T = T_1\cup T_2 \cup \cdots \cup T_\ell$ be a forest, where $V(T_i)=A_i^0\cup A_i^1$  for $i\in [\ell]$.
    For any $\beta >0$, there exists a constant $C$ and sufficiently large integer $n$ such that the following holds:   
    Let $G$ be an $n$-vertex graph and $X, Y, Z_1, \ldots, Z_{r-3}$ be pairwise disjoint subsets of $V(G)$, with $\chi(G[X])\ge C$ and $|Y|=|Z_j|$ for each $j \in[r-3]$. Suppose that $(Y, Z_j)$ and $(Z_i, Z_j)$ are $(\varepsilon, d)$-regular for each $i \neq j$, and that
\begin{align*}
|N(x) \cap Y| \geq \beta|Y| \quad \text{and} \quad |N(x) \cap Z_j| \geq (1/2+\beta)|Z_j|
\end{align*}
for every $x \in X$ and $j \in[r-3]$.
Then $Z_{\ell}^{r, t}(T_1, \ldots, T_{\ell})$ can be embedded into $G$ in the following way:
there exists a copy of $T$ in $G[X]$, disjoint subsets $S_I$ of size $t$ for $I\subseteq [\ell]$, and $S'_j$ of size $t$ for $j\in [r-3]$ such that
\begin{itemize}
    \item for each $I\subseteq [\ell]$, $S_I\subseteq N(A^{\bm{u}(I)}) \cap Y$;
    \item for each $j\in [r-3]$, $S'_j\subseteq N(T)\cap Z_j$, and $S'_j$ is completely adjacent to each $S'_{j^{\prime}}$ with $j \neq j^{\prime}$ and also to each $S_I$ with $I\subseteq [\ell]$. 
\end{itemize}
\end{lemma}

  Given a graph $G$, a set $Y \subseteq V(G)$, and integers $\ell, t \in \mathbb{N}$, define $\mathcal{G}_{\ell}^{r,t}(Y)$ to be the collection of functions
  $$S: 2^{[\ell]} \cup[r-3] \rightarrow\binom{Y}{t}.$$

  We say that $S=(S_I : I\subseteq [\ell], S’_1,\dots, S’_{r-3})  \in \mathcal{G}_{\ell}^{r,t}(Y)$ is \textit{proper} if these sets are pairwise disjoint.
  We shall write $\mathcal{F}_{\ell}^{r,t}(Y)$ for the collection of proper functions in $\mathcal{G}_{\ell}^{r,t}(Y)$. 
  For any $S\in \mathcal{F}_{\ell}^{r,t}(Y)$, it is natural to think of $S$ as a family $\{S_I: I \subseteq[\ell]\} \cup\{S_j^{\prime}: j \in[r-3]\}$ of disjoint subsets of $Y$ of size $t$. 

  For an ordered pair $(x, y)$ of vertices of $G$, a function $S \in \mathcal{F}_{\ell}^{r,t}(Y)$, and $i \in[\ell]$, we write $(x, y) \rightarrow_i S$, if the following conditions hold:
\begin{itemize}
    \item $S_j' \subseteq N(x, y) \text { for every } j \in[r-3]$;
    \item $\bigcup\limits_{I: i \in I} S_I \subseteq N(x)  \text { and } \bigcup\limits_{I: i \notin I} S_I \subseteq N(y).$
\end{itemize}

For an edge $e=x y \in E(G)$, we write $e \rightarrow_i S$ if either $(x, y) \rightarrow_i S$ or $(y, x) \rightarrow_i S$. 
When edges $e_1,\dots, e_{\ell}$ are given, we write $\mathbf{e}^j$ to denote the $j$-tuple $(e_1,\dots, e_j)$ for each $j\in [\ell]\cup \{0\}$, with $\mathbf{e}^0$ representing the empty tuple.
We define
$$
\mathbf{e}^{\ell} \rightarrow S \Leftrightarrow e_i \rightarrow_i S \quad \text { for each } i \in[\ell] .
$$

Note that the graph $Z_{\ell}^{r, t}$ consists of a set of pairwise disjoint edges $e_1, \ldots, e_{\ell}$ and a function $S \in \mathcal{F}_{\ell}^{r, t}(Y)$ such that $\mathbf{e}^{\ell} \rightarrow S$.

  Recall that for any subset $X \subseteq V(G)$, $E(X)$ denotes the edge set of the subgraph $G[X]$.
  If $D \subseteq E(G)$, then $\delta(D)$ ($\bar{d}(D)$) represents the minimum degree (average degree, respectively) of the graph induced by $D$, that is the graph with vertex set $\bigcup_{e\in D} e$ and edge set $D$.

  In the following definition, the sets $D(\mathbf{e}^j)$ should be interpreted as forming a hierarchical structure (or tree) of graphs.
  Specifically, each edge of $D(\mathbf{e}^j)$ is associated with a distinct graph $D(\mathbf{e}^{j+1})$.
  This hierarchy satisfies the following key properties: 
\begin{itemize}
    \item[\rm (a)] each graph $D(\mathbf{e}^j)$ has large minimum degree, and
    \item[\rm (b)] for a positive proportion of the sets $S \in \mathcal{F}_{\ell}^{r,t}(Y)$, we have $\mathbf{e}^{\ell} \rightarrow S$.
\end{itemize}  

\begin{definition}[$(C, \alpha)$-Rich in Copies of $Z_{\ell}^{r,t}$]
  Let $X$ and $Y$ be disjoint vertex sets in a graph $G$, let $C \in \mathbb{N}$ and $\alpha>0$, and let $s:=2^{\ell}t$. We say that $(X, Y)$ is $(C, \alpha)$-rich in copies of $Z_{\ell}^{r,t}$ if
$$
\begin{aligned}
& \exists D= D(\mathbf{e}^0) \subseteq E(X) \text{\quad s.t.} \\
&\qquad\forall e_1 \in D,~\exists D(\mathbf{e}^1) \subseteq E(X) \text{\quad s.t.} \\
&\qquad\qquad\qquad\qquad\qquad \vdots \\
&\qquad\qquad \qquad  \forall e_{\ell-1} \in D(\mathbf{e}^{\ell-2}),~ \exists D(\mathbf{e}^{\ell-1}) \subseteq E(X) \text{\quad s.t. \quad} \forall e_{\ell} \in D(\mathbf{e}^{\ell-1})
\end{aligned}
$$
the following properties hold:
\begin{itemize}
    \item[\rm (a)] $\delta(D), \delta(D(\mathbf{e}^1)), \ldots, \delta\left(D\left(\mathbf{e}^{\ell-1}\right)\right)>C$, and
    \item[\rm (b)] $|\{S \in \mathcal{F}_{\ell}^{r,t}(Y): \mathbf{e}^{\ell} \rightarrow S\}| \geq \alpha|Y|^s$.
\end{itemize}

\end{definition}

\begin{definition}[Good Function, $(C, \alpha)$-Dense, \cite{allen2013chromatic}]
  A function $S \in \mathcal{F}_{\ell}^{r,t}(Y)$ is $(\ell, t, C, \alpha)$-good for a tuple $\mathbf{e}^q$ and $(X, Y)$ if there exist sets
$$
E_{q+1}, \ldots, E_{\ell} \subseteq E(X), \quad \text { with } \bar{d}(E_j) \geq 2^{-\ell} \alpha C \text { for each } q+1 \leq j \leq \ell,
$$
such that for every $e_{q+1} \in E_{q+1}, \ldots, e_{\ell} \in E_{\ell}$, we have $\mathbf{e}^{\ell} \rightarrow S$.
\end{definition}

When the constants $(\ell, t, C, \alpha)$ and the sets $(X, Y)$ are clear from the context, we shall omit them. We shall abbreviate ``$(\ell, t, C, \alpha)$-good for $\mathbf{e}^0$ and $(X, Y)$ '' to ``$(\ell, t, C, \alpha)$-good for $(X, Y)$''.

The pair $(X, Y)$ is \emph{$(C, \alpha)$-dense} in copies of $Z_{\ell}^{r,t}$ if there exist at least $2^{-\ell} \alpha|Y|^s$ families $S \in  \mathcal{F}_{\ell}^{r,t}(Y)$ which are $(\ell, t, C, \alpha)$-good for $(X, Y)$.

\begin{proposition}[\cite{allen2013chromatic}]\label{prop:prop26}
    For every $\ell, t \in \mathbb{N}$ and $\beta>0$, there exists $\alpha>0$ such that, for every $C \in \mathbb{N}$, there exists $C^{\prime} \in \mathbb{N}$ such that the following holds. Let $G$ be a graph and let $X$ and $Y$ be disjoint subsets of $V(G)$, such that $|N(x) \cap Y| \geq \beta|Y|$ for every $x \in X$.

Then either $\chi(G[X]) \leq C^{\prime}$, or $(X, Y)$ is $(C, \alpha)$-rich in copies of $Z_{\ell}^{3, t}$.
\end{proposition}

\begin{lemma}[\cite{allen2013chromatic}]\label{lmm24}
 If $(X, Y)$ is $(C, \alpha)$-rich in copies of $Z_{\ell}^{r,t}$, then $(X, Y)$ is $(C, \alpha)$-dense in copies of $Z_{\ell}^{r,t}$. 
\end{lemma}


Combining \cref{prop:prop26} and \cref{lmm24}, we can substitute the condition ``$(X, Y)$ is $(C^*, \alpha)$-dense in copies of $Z_{\ell^*}^{3, t^*}$'' of Proposition 36 in \cite{allen2013chromatic} with  ``$|N(x) \cap Y| \geq \beta|Y|$ for every $x \in X$''.
We will use the proof of Proposition 36 in \cite{allen2013chromatic} to prove \cref{lmm:0-1/2}, omitting the repetitive parts.

\begin{proof}[Proof of Lemma \ref{lmm:0-1/2}]
   Let $G$ be a graph, and let $X,Y$ and $Z_1,\ldots, Z_{r-3}$ be disjoint subsets of $V(G)$, with $\chi(G[X])\ge C$ and $|Y|=|Z_j|$ for each $j \in[r-3]$.
   Let $Z:=Z_1 \cup \cdots \cup Z_{r-3}$. Suppose that $(Y, Z_j)$ and $(Z_i, Z_j)$ are $(\varepsilon, d)$-regular for each $i \neq j$, and that
\begin{align*}
|N(x) \cap Y| \geq \beta|Y| \quad \text{and} \quad |N(x) \cap Z_j| \geq (1/2+\beta)|Z_j|
\end{align*}
for every $x \in X$ and $j \in[r-3]$.

According to Lemma 37 and Claim 39 in \cite{allen2013chromatic}, there exists a copy $Q$ of $K_{r-3}[t]$ with $t$ vertices in $Z_j$ for each $j \in[r-3]$ and a function $S \in \mathcal{F}_{\ell}^{r, t}(Y\cup Z)$ such that the following holds.
\begin{itemize}
    \item The sets $S_j, j \in[r-3]$, are the parts of $Q$;
    \item for each $I \subseteq [\ell]$, $S_I\subseteq N(Q) \cap Y$ with $|S_I|=t$;
    \item there exist sets $E_1, \ldots, E_{\ell} \subseteq E(X)$, with $\bar{d}\left(E_i\right) \geq 8|T|$ for each $1 \leq i \leq \ell$, such that for every $e_1 \in E_1, \ldots, e_{\ell} \in E_{\ell}$, we have $\mathbf{e}^{\ell} \rightarrow S$;
    \item $S_j$ is completely adjacent to each $S_{j^{\prime}}$ with $j \neq j^{\prime}$, to each $S_I$, and to each edge $e \in \bigcup_{i \in [\ell]} E_i$. 
\end{itemize}

Now we can embed $Z_\ell^{r,t}(T_1,\ldots,T_\ell)$ via hierarchical selection.
For each $i \in[\ell]$ and each edge $e \in E_i$, let $e=x y$ be such that $(x, y) \rightarrow_i S$, and orient the edge $e$ from $x$ to $y$. 
 For each $i \in[\ell]$, by choosing a maximal bipartite subgraph of $E_i$, and then removing at most half the edges, we can find a set $E_i^{\prime} \subseteq E_i$ such that,
\begin{itemize}
    \item[\rm (a)] $E_i'$ is bipartite, with bipartition $A_i^0\cup A_i^1$,
    \item[\rm (b)] every edge $e \in E_i'$ is oriented from $A_i^0$ to $A_i^1$, and
    \item[\rm (c)] $\bar{d}(E_i') \geq 2 |T|$.
\end{itemize}
Thus, by \cref{fact:4}, there exists, for each $i \in[\ell]$, a copy $T_i^{\prime}$ of $T_i$ in the induced subgraph
  $$E_i'-(V(T_1') \cup \cdots \cup V(T_{i-1}')),$$
since removing a vertex can only decrease the average degree by at most two.
Hence, by the definition of $S \in \mathcal{F}_{\ell}^{r, t}(Y\cup Z)$, 
these trees together with $S_I$ for each $I\subseteq [\ell]$ and $S_j$ for each $j\in[r-3]$, form a copy of $Z_\ell^{r,t}(T_1,\ldots,T_\ell)$ in $G$, so we are done.
\end{proof}

\end{document}